\documentclass[11pt,twoside,a4paper]{amsart}

\usepackage{amsmath}
\usepackage{amsthm}
\usepackage{color}
\usepackage{mathrsfs}

\usepackage{amsfonts}
\usepackage{amssymb}
\usepackage{amsmath}
\usepackage{amsthm}
\usepackage{bbm}
\usepackage{verbatim}
\usepackage{url}

\allowdisplaybreaks

\DeclareMathAlphabet{\mathpzc}{OT1}{pzc}{m}{it}

\makeatletter
\DeclareFontFamily{OMX}{MnSymbolE}{}
\DeclareSymbolFont{MnLargeSymbols}{OMX}{MnSymbolE}{m}{n}
\SetSymbolFont{MnLargeSymbols}{bold}{OMX}{MnSymbolE}{b}{n}
\DeclareFontShape{OMX}{MnSymbolE}{m}{n}{
    <-6>  MnSymbolE5
   <6-7>  MnSymbolE6
   <7-8>  MnSymbolE7
   <8-9>  MnSymbolE8
   <9-10> MnSymbolE9
  <10-12> MnSymbolE10
  <12->   MnSymbolE12
}{}
\DeclareFontShape{OMX}{MnSymbolE}{b}{n}{
    <-6>  MnSymbolE-Bold5
   <6-7>  MnSymbolE-Bold6
   <7-8>  MnSymbolE-Bold7
   <8-9>  MnSymbolE-Bold8
   <9-10> MnSymbolE-Bold9
  <10-12> MnSymbolE-Bold10
  <12->   MnSymbolE-Bold12
}{}

\let\llangle\@undefined
\let\rrangle\@undefined
\DeclareMathDelimiter{\llangle}{\mathopen}%
                     {MnLargeSymbols}{'164}{MnLargeSymbols}{'164}
\DeclareMathDelimiter{\rrangle}{\mathclose}%
                     {MnLargeSymbols}{'171}{MnLargeSymbols}{'171}
\makeatother

\def\rr{{\mathbb R}}
\def\rd{{{\rr}^d}}

\def\zz{{\mathbb Z}}
\def\nn{{\mathbb N}}

\def\cm{{\mathcal M}}

\def\fz{\infty}

\def\supp{{\mathop\mathrm{\,supp\,}}}

\def\lz{\lambda}

\def\wz{\widetilde}
\def\gfz{\genfrac{}{}{0pt}{}}

\def\ls{\lesssim}

\def\lo{{L^1(\rd)}}

\def\lin{{L^\fz(\rd)}}

\newtheorem{thm}{Theorem}[section]
\newtheorem{lem}[thm]{Lemma}
\newtheorem{prop}[thm]{Proposition}
\newtheorem{rem}[thm]{Remark}
\newtheorem{defn}[thm]{Definition}

\DeclareMathOperator*{\esssup}{ess\ sup}

\numberwithin{equation}{section}

\begin{document}

\arraycolsep=1pt

\title[Variation with Matrix Weight]{Variation of Calder\'on--Zygmund Operators with\\ Matrix Weight}
\author{Xuan Thinh Duong, Ji Li and Dongyong Yang}

\date{}
\maketitle

\begin{center}
\begin{minipage}{13.5cm}\small

{\noindent  {\bf Abstract:}\ Let $p\in(1,\infty)$, $\rho\in (2, \infty)$ and $W$ be a matrix $A_p$ weight.  In this article, we introduce a version of variation $\mathcal{V}_{\rho}({\mathcal T_n}_{\,,\,\ast})$ for matrix Calder\'on--Zygmund operators
 with modulus of continuity satisfying the Dini condition. We then obtain
the $L^p(W)$-boundedness of $\mathcal{V}_{\rho}({\mathcal T_n}_{\,,\,\ast})$ with norm
\begin{align*}
\|\mathcal{V}_{\rho}({\mathcal T_n}_{\,,\,\ast})\|_{L^p(W)\to L^p(W)}\leq  C[W]_{A_p}^{1+{1\over p-1} -{1\over p}}
\end{align*}
by first proving a sparse domination of the variation of the scalar Calder\'on--Zygmund operator, and then providing a convex body sparse domination of the variation of the matrix Calder\'on--Zygmund operator. The key step here is a weak type estimate of a local grand maximal truncated operator with respect to the scalar Calder\'on--Zygmund operator.
}
\end{minipage}
\end{center}

%
\bigskip

{\small {\it Keywords}: Calder\'on--Zygmund operator, variation, matrix weight, sparse operator.}

\medskip

{\small{Mathematics Subject Classification 2010:} {42B20, 42B99}}


\section{Introduction and Statement of Main Results\label{s1}}

It is well known that scalar Muckenhoupt $A_p$ weights have a long history since 1970s and are central to the study of weighted norm inequalities in modern harmonic analysis.
Matrix $A_p$ weights are more recent, introduced  by Nazarov, Treil, Volberg \cite{NazarovTreil96AlgeAnal}, \cite{Volberg97JAMS},
\cite{TreilVolberg97JFA}, \cite{NazarovTV99JAMS},  and arose from problems in stationary processes and operator theory. And later harmonic analysis with matrix weights have been intensively studied by many authors, see for example
 \cite{Roudenko02PHDT}, \cite{Goldberg03PacificJM}, \cite{BickelPW16JFA}, \cite{BickelWick16JMAA},
\cite{Cruz-UribeMR16JGA}, \cite{BickelLM17JMAA},
\cite{PottStoica17BullSM},
\cite{IsralowitzKP17JLMS}, \cite{NazarovPTV17AdvMath} and so on. Among these, we would like to mention the recent notable result of Nazarov, Petermichl, Treil and Volberg \cite{NazarovPTV17AdvMath}, where they introduced the so-called convex body valued sparse operator, which generalizes the notion of sparse operators in the scalar setting (\cite{LernerOR17AdvMath}) to the case of space of vector valued functions. And then they proved the domination of Calder\'on--Zygmund operators by such sparse operators, and hence by estimating sparse operators they obtained the weighted estimates with matrix weights, which in turn yields the weighted estimates for Calder\'on--Zygmund operators with matrix weights.

In \cite{Lepingle76ZWVG}, L\'epingle first proved a variation inequality for martingales which improves the classical Doob maximal inequality (see also \cite{PisierXu88PTRF}). Based on L\'epingle's result, Bourgain \cite{Bourgain89IHSPM} further obtained corresponding variational estimates for the Birkhoff ergodic averages along subsequences of natural numbers and
then directly deduced pointwise convergence results without previous knowledge that pointwise convergence holds for a dense subclass of functions, which are not available in some ergodic models. Since then, the variational inequalities have been the subject of many recent articles
in probability, ergodic theory and harmonic analysis.
In particular, Campbell et al. \cite{CampbellJRW00DMJ}
established the $L^p$-boundedness of variation for truncated
Hilbert transform and then
extended to higher dimensional case in \cite{CampbellJRW03TAMS}.
Then in \cite{GillespieTorrea04IsraelJM}, to  obtain 
 dimension-free estimate for the oscillation and variation of the Riesz transforms operating on
 $L^p (\mathbb R^d, |x|^\alpha)$, Gillespie and Torrea further established 
some $A_p$-weighted norm inequalities for the oscillation
and the variation of the Hilbert transform in $L^p$ for $p\in(1, \infty)$.
In \cite{MaTX15JFA}, Ma et al. established the weighted norm inequalities for the variation of 
Calder\'on-Zygmund operators on $\mathbb R$ for $A_p(\mathbb R)$ weights with $p\in [1, \infty).$
For more results on variational inequalities and their applications, see, for example, \cite{GillespieTorrea04IsraelJM,Jones06ComtempMath,BetancorFHR13PA,
Bui14NonlinearAnal,MasTolsa14JEMS,
MaTX15JFA,MasTolsa17TAMS,MirekST17InventMath,BourgainMSW18GFA} and references therein.

In this paper, we study the weighted estimates for the variation for Calder\'on--Zygmund operators whose kernels satisfy the 
 Dini condition (see Definition \ref{def CZO} below) with matrix weights.
To this end, we recall a few necessary notation and definitions in the setting of matrix weight.

Let $n$ and $d$ be natural numbers and $W: \mathbb{R}^d\to \mathcal{M}_n(\mathbb{R})$ be positive definite a.\,e. (where as usual
$\mathcal{M}_n(\mathbb{R})$ is the algebra of $n\times n$ matrices with complex scalar entries). We say $W$ is measurable if each component of $W$ is a measurable function. Recall that if $W$ is a self-adjoint and positive definite
matrix, then it has $n$ non-negative eigenvalues $\lz_i$, $i\in\{1,\cdots, n\}$.
Moreover, there exists a measurable orthogonal matrix function $U$
such that $U^t WU=D(\lz_1,\cdots,\lz_n)=: D(\lz_i)$ is diagonal (see \cite[Lemma 2.3.5]{RonShen95CJM}). Now for
any $s>0$, define $W^s:=UD(\lz^s_i)U^t$. For a fixed matrix function $W$,
we will always implicitly assume that all of its powers are defined using the same
orthogonal matrix $U$. For such matrix function, we then define $L^p(W)$ for $1<p<\infty$ to be the space of measurable functions $\vec{f}: \mathbb{R}^d\to \mathbb{R}^n$ with norm
$$\left\|\vec{f}\right\|_{L^p(W)}:=\bigg( \int_{\mathbb{R}^d} \Big| W^{1\over p}(x) \vec{f}(x) \Big|^p\,dx \bigg)^{1/p}<\fz.
$$

By a matrix weight we mean a matrix function $W$ such that
$\|W\|\in L^1_{\rm loc}(\rd, \rr)$, where for each $x\in\rd$,
$$\|W(x)\|:=\sup_{\vec e\in \mathbb R^n,\,|\vec e|=1}|W(x)\vec e|.$$
For $s>0$, we also define negative powers of $W$ through the orthogonal matrix $U$
by setting $W^{-s}:= UD(\lz_i^{-s})U^t$. Now suppose $1<p<\infty$. We say $W$ is a matrix $A_p$ weight (and we write $W\in \mathcal A_p$ in this case), if
\begin{align}\label{e-matrix ap defn}
 [W]_{\mathcal A_p}:=\sup_{Q\subset \mathbb{R}^d,\,Q\, {\rm is\ a\ cube }} {1\over |Q|} \int_Q \bigg( {1\over |Q|} \int_Q
\Big\| W^{1\over p}(x)  W^{-{1\over p}}(t)\Big\|^{p'} dt \bigg)^{p\over p'} dx<\infty,
\end{align}
where $p'$ is the conjugate index of $p$.

We now recall the Calder\'on--Zygmund operator as follows.
\begin{defn}\label{def CZO}
$T$ is a scalar Calder\'on--Zygmund operator with kernel $K(x,y)$ defined on $\mathbb{R}^d\times \mathbb{R}^d\backslash \{(x,y):\ x=y\}$ if
$T$ is bounded on $L^2(\mathbb{R}^d, \rr)$, and the kernel $K(x,y)$  satisfies the following size and smoothness condition:
\begin{itemize}
\item  for any $x, y$ with $x\not=y$,
\begin{align}\label{e-cz kernel size cond}
|K(x,y)|\leq {\displaystyle C\over \displaystyle |x-y|^d};
\end{align}
\item  for any $x,\, x',\,y$ with $|x-x'|\leq |x-y|/2$,
\begin{align}\label{e-cz kernel smooth cond}
 |K(x,y)-K(x',y)|+  |K(y,x)-K(y,x')|\leq  C\omega\bigg({\displaystyle |x-x'| \over \displaystyle |x-y|}\bigg){\displaystyle1 \over \displaystyle |x-y|^{d}},
 \end{align}
 \end{itemize}
 where $\omega(t)$ is an increasing subadditive function on $[0,\infty)$ satisfying $\omega(0)=0$ and the Dini condition
\begin{align}\label{e-dini cond}
 \int_0^1 w(t) {dt\over t}<\infty.
\end{align}
\end{defn}

\begin{rem}\label{r-czo}
In \cite{MaTX15JFA}, the kernel $K(x, y)$ of a Calder\'on-Zygmund operator $T$ is assumed to 
be a function on $\mathbb R\times \mathbb R\setminus \{(x,x):\, x\in\mathbb R\}$ satisfying the following conditions: there exist positive constants $C$ and $\delta$ such that
\begin{itemize}
  \item [(i)] for any $x\not= y$, 
  \begin{align*}
  |K(x, y)|\le \frac C{|x-y|};
  \end{align*}
  \item [(ii)] for any $x,\, x',\,y$ with $|x-x'|\leq |x-y|/2$,
\begin{align}\label{e-cz kernel regularity}
 |K(x,y)-K(x',y)|+  |K(y,x)-K(y,x')|\leq  C\frac{|x-x'|}{|x-y|^{1+\delta}}.
 \end{align}
\end{itemize}
We then see that when $d=1$, if a kernel $K$  satisfies 
\eqref{e-cz kernel regularity}, then $K$ also satisfies \eqref{e-cz kernel smooth cond}.
\end{rem}

Let $T$ be a scalar Calder\'on--Zygmund operator as above, and for any $\epsilon\in (0, \fz)$,
$$T_\epsilon f(x):=
\int_{|x-y|>\epsilon} K(x, y)f(y)\,dy.$$
 Recall that variation operator $\mathcal{V}_{\rho}({\mathcal T}_{\ast}f)$ for $\{T_\epsilon\}$  and $\rho\in(2, \fz)$ is defined by
 \begin{equation}\label{def variation}
 \mathcal{V}_{\rho}({\mathcal T}_{\ast}f)(x):= \sup_{\epsilon_i\searrow0}
 \Big(\sum_{i=1}^{\infty}|T_{\epsilon_{i+1}}f(x)-T_{\epsilon_{i}}f(x)|^{\rho}\Big)^{1/\rho},
 \end{equation}
 where the supremum is taken over all sequences $ \{\epsilon_{i}\}$ decreasing to zero.

Then our first main result is the sparse domination of $\mathcal{V}_{\rho}({\mathcal T}_{\ast}f)$.

To be more precise, let $\mathcal F$ and $\overline{\mathcal F}$ be two collections of disjoint
dyadic cubes. We say $\overline {\mathcal F}$ covers $\mathcal F$ if for any
cube $Q\in \mathcal F$, one can find $R\in \overline{\mathcal F}$ such that $Q\subset R.$ For a given cube $Q_0\subset \rd$, let $\mathcal D(Q_0)$ denote the set of all dyadic cubes with respect to $Q_0$, that is, the  cubes obtained by repeated subdivision of $Q_0$ and each of its descendants in $2^d$ congruent subcubes. Next we establish the following pointwise sparse domination on the variation operator $\mathcal{V}_{\rho}({\mathcal T}_{\ast}f)(x)$ (see  \cite{Lerner16NewYork} for a similar version of sparse domination of $T$).

\begin{thm}\label{p-sparse domi}
Let $T$ be a Calder\'on--Zygmund operator as in Definition \ref{def CZO} and $\varepsilon\in(0, 1)$. Suppose that  $\mathcal{V}_{\rho}({\mathcal T}_{\ast})$ is bounded on
$L^q(\mathbb R^d, \mathbb R)$ for some $q\in (1, \fz)$.   For every $f\in L^\infty(\rd)$ with  $\supp(f)\subset Q_0$,
there exists a family $\mathcal F\subset \mathcal D(Q_0)$
of disjoint dyadic subcubes of $Q_0$ such that
\begin{itemize}
  \item [{\rm (i)}]$\sum_{Q\in \mathcal F}|Q|\le \varepsilon |Q_0|$;
  \item [{\rm (ii)}] for a.\,e. $x\in Q_0$,
\begin{align*}
\left|\mathcal{V}_{\rho}({\mathcal T}_{\ast}f)(x)-\sum_{Q\in \mathcal F}\mathcal{V}_{\rho}({\mathcal T}_{\ast}(f\chi_{3Q}))(x)\chi_Q(x)\right|\leq C
|f|_{3Q_0};
\end{align*}
\item [{\rm (iii)}] for any collection $\overline{\mathcal F}$ of disjoint dyadic subcubes of $Q_0$ that covers $\mathcal F$,
    \begin{align*}
\left|\mathcal{V}_{\rho}({\mathcal T}_{\ast}f)(x)-\sum_{Q\in \overline{\mathcal F}}\mathcal{V}_{\rho}({\mathcal T}_{\ast}(f\chi_{3Q}))(x)\chi_Q(x)\right|\leq C
|f|_{3Q_0}.
\end{align*}
\end{itemize}
\end{thm}

We point out that the key step for this sparse domination is the weak type $(1,1)$ estimate of a local grand maximal truncated operator  $\mathcal M_{\mathcal{V}_{\rho}({\mathcal T}_{\ast}),Q_0}$
defined as below (see also \cite{Lerner16NewYork}).
Given a cube $Q_0\subset \rd$,  $\mathcal M_{\mathcal{V}_{\rho}({\mathcal T}_{\ast}),Q_0}$ is defined as follows:
\begin{align}\label{Gmaximal}
\mathcal M_{\mathcal{V}_{\rho}({\mathcal T}_{\ast}),Q_0}f(x) := \left\{
                                                                  \begin{array}{ll}
                                                                    \sup\limits_{Q\ni x,\, Q\subset Q_0} \esssup\limits_{\xi\in Q} \mathcal{V}_{\rho}\big({\mathcal T}_{\ast}(f\chi_{3Q_0\setminus3Q})\big)(\xi), & {\ \ \ x\in Q_0;} \\[15pt]
                                                                    0, & \hbox{\ \ \ otherwise.}
                                                                  \end{array}
                                                                \right.
\end{align}
For the weak type  boundedness of
$\mathcal M_{\mathcal{V}_{\rho}({\mathcal T}_{\ast}),Q_0}(f)$, we refer to Proposition \ref{p-grand maxi truncat bdd}.

Next we denote by $Id_n$ the $n$-th order identity matrix.
Now we define
\begin{align}\label{Tn}
 T_n:= T\otimes Id_n.
\end{align}
To be more specific, suppose $\vec{f}:\  \mathbb{R}^d\to \mathbb{R}^n$  is a vector-valued function, say $\vec{f}:=(f_1,\ldots,f_n)$. Then the action of $T_n$ on the vector is componentwise, i.\,e.,
$$ T_n(\vec{f}):= ( Tf_1,\ldots, Tf_n ). $$

We now introduce the variation for the matrix Calder\'on--Zygmund operator $T_n$ as follows.
\begin{defn}
Let $T_n$ be a matrix Calder\'on--Zygmund operator as defined in \eqref{Tn}. We define
\begin{align*}
\mathcal{V}_{\rho}({\mathcal T_n}_{\,,\,\ast}\vec f\ )(x) :=
\begin{bmatrix}
      \mathcal{V}_{\rho}({\mathcal T}_{\ast}f_1)(x)     \\
       \mathcal{V}_{\rho}({\mathcal T}_{\ast}f_2)(x)  \\
      \vdots  \\
      \vdots\\
    \mathcal{V}_{\rho}({\mathcal T}_{\ast}f_n)(x)
\end{bmatrix},
\end{align*}
where $\vec{f}:=(f_1,\ldots,f_n)$.
\end{defn}
Given $q\in(1,\infty)$, we say that $\mathcal{V}_{\rho}({\mathcal T_n}_{\,,\,\ast})$ is bounded on $L^q(\mathbb R^d,\mathbb R^n)$, if there exists a positive constant $C$ such that for any $\vec f\in L^q(\rd, \mathbb R^n)$,
$$\int_\rd \left|\mathcal{V}_{\rho}({\mathcal T_n}_{\,,\,\ast}\ \vec f)(x)\right|^q\,dx\le C\int_\rd\left|\vec f(x)\right|^q\,dx.$$
Note that this is equivalent to the fact that $\mathcal{V}_{\rho}({\mathcal T}_{\ast})$ is bounded on $L^q(\rd, \mathbb R)$. Now we state our main result in this paper.

\begin{thm}\label{t-lp bdd vari matrix weigh}
Suppose $p\in(1, \infty)$, $\rho\in (2, \fz)$ and $W$ is a matrix $A_p$ weight.  Let $T$ be a scalar Calder\'on--Zygmund operator
and $T_n$ is defined as in \eqref{Tn}. Suppose $\mathcal{V}_{\rho}({\mathcal T_n}_{\,,\,\ast})$ is bounded on
$L^q(\mathbb R^d,\mathbb R^n)$ for some $q\in(1,\infty)$. Then $\mathcal{V}_{\rho}({\mathcal T_n}_{\,,\,\ast})$ is bounded on $L^p(W)$ with
\begin{align*}
\left\|\mathcal{V}_{\rho}\left({\mathcal T_n}_{\,,\,\ast} \right)\right\|_{L^p(W)\to L^p(W)}\leq  C[W]_{A_p}^{1+{1\over p-1} -{1\over p}}.
\end{align*}
\end{thm}

In Section \ref{s2} we prove our first main result Theorem \ref{p-sparse domi}, i.e.,  the pointwise sparse domination on the variation operator $\mathcal{V}_{\rho}({\mathcal T}_{\ast}f)(x)$. The key step for this sparse domination is weak type $(1,1)$ estimate of
$\mathcal M_{\mathcal{V}_{\rho}({\mathcal T}_{\ast}),Q_0}f(x)$ defined in \eqref{Gmaximal}. Compared
to the local grand maximal truncated operator of $T$ considered in \cite{Lerner16NewYork}, our
proof here is more complicated.
In Section \ref{s3}, we first
extend the pointwise domination of variation in Section \ref{s2} to the vector-valued setting. Then we obtain a version of domination of $\mathcal{V}_{\rho}({\mathcal T_n}_{\,,\,\ast}\vec f)$ by convex body valued sparse operators introduced in \cite{NazarovPTV17AdvMath}, and by following some
idea from \cite{Cruz-UribeIM18IEOT}, present the proof of Theorem \ref{t-lp bdd vari matrix weigh}.

Throughout this paper, we write $A\ls B$ if $A\le C B$,
 where $C$ is a positive constant whose value may change from line to line, and the symbol $A\approx B$ as usual means that $A\ls B$ and $B\ls A$.
  For any cube $Q\subset \rd$ and any scalar function $f\in L^1_{\rm loc}(\rd,\rr)$,  $f_Q$ means the mean value of $f$ over $Q$. For any $t>0$,
$tQ$ is the cube concentric with $Q$ and has side-length $tl(Q)$.
For a given measurable subset $E$ of $\rd$, $\chi_E$ means the characteristic function of $E$. Finally, for $p\in[1, \fz]$, we write the notation $L^p(\rd, \rr)$ simply by $L^p(\rd)$.

\section{Weak type estiamte and sparse domination of variation operator: \\ proof of Theorem \ref{p-sparse domi}}\label{s2}
In this section, we first establish the weak type (1, 1) estimate of variation
$\mathcal{V}_{\rho}({\mathcal T}_{\ast})$ and the local grand maximal truncated operator  $\mathcal M_{\mathcal{V}_{\rho}({\mathcal T}_{\ast}),Q_0}$ , and then we obtain a version of the pointwise domination on $\mathcal{V}_{\rho}({\mathcal T}_{\ast}f)$ for suitable functions $f$.

\begin{prop}\label{p-vari lp bdd imp weal type}
Assume that $T$ is as in Definition \ref{def CZO} and $\mathcal{V}_{\rho}({\mathcal T}_{\ast})$ is bounded on
$L^q(\mathbb R^d, \mathbb R)$ for some $q\in (1, \fz)$. Then $\mathcal{V}_{\rho}({\mathcal T}_{\ast})$
is of weak type (1,1).
\end{prop}

\begin{proof}
The proof follows from the one developed in \cite{CampbellJRW00DMJ, CampbellJRW03TAMS} for Hilbert transform and for Calder\'on--Zygmund
operators in higher dimension.
For any $\lz>0$ and function $f\in \lo$, applying
the Calder\'on--Zygmund decomposition to $f$ at height $\lz$ (see, for example,
\cite[Theorem 4.3.1]{Grafakos08}), we then have
$f=g+b$ such that
\begin{itemize}
  \item [${\rm (c_1)}$] $|g(x)|\le 2^d\lz$ for a. e. $x\in\rd$ and $\|g\|_\lo\le\|f\|_\lo$;
  \item [${\rm (c_2)}$] $b=\sum_j b_j$ such that for each $j$,
  $\supp(b_j)\subset Q_j$ and $\{Q_j\}\subset \mathcal D(\rd)$
 is a mutually disjoint sequence of cubes, where $\mathcal D(\rd)$ is the family of  dyadic cubes in $\rd$;
  \item [${\rm (c_3)}$] for each $j$, $\int_\rd b_j(x)\,dx=0$;
  \item [${\rm (c_4)}$] for each $j$, $\|b_j\|_\lo\le 2^{d+1} \lz|Q_j|$;
  \item [${\rm (c_5)}$] $\sum_j |Q_j|\le \frac1\lz\|f\|_\lo$.
\end{itemize}

By the sublinearity of $\mathcal{V}_{\rho}({\mathcal T}_{\ast})$,
we first write
$$\left|\left\{x\in \rd:\,\, \mathcal{V}_{\rho}({\mathcal T}_{\ast}f)(x)>\lz\right\}\right|\le \left|\left\{x\in \rd:\,\, \mathcal{V}_{\rho}({\mathcal T}_{\ast}g)(x)>\frac\lz2\right\}\right|+\left|\left\{x\in \rd:\,\, \mathcal{V}_{\rho}({\mathcal T}_{\ast}b)(x)>\frac\lz2\right\}\right|.$$
Using the $L^q(\rd)$-boundedness of $\mathcal{V}_{\rho}({\mathcal T}_{\ast})$
and ${(c_1)}$, we see that
$$\left|\left\{x\in \rd: \,\, \mathcal{V}_{\rho}({\mathcal T}_{\ast}g) (x)>\frac\lz2\right\}\right|\ls \frac1\lz \|f\|_\lo.$$
Let $\widetilde Q:=\cup_j 9\sqrt dQ_j$. Then ${(c_5)}$ implies
$|\widetilde Q|\ls \frac1\lz\|f\|_\lo.$
Thus, it remains to estimate $\mathcal{V}_{\rho}({\mathcal T}_{\ast}b)(x)$ for $x\in \rd\setminus \widetilde Q$. Observe that $T$ is of weak type (1,1) (see \cite[p.\,192]{Grafakos09}). For every $x\in \rd\setminus \widetilde Q$, by $(c_2)$ and the definition of
$\mathcal{V}_{\rho}({\mathcal T}_{\ast}b)(x)$, we can take a decreasing sequence $\{t_i\}$
such that $t_i\to0$ as $i\to \fz$, and
$$\mathcal{V}_{\rho}({\mathcal T}_{\ast}b)(x)\le
2 \left(\sum_i\left|\sum_j T\left(\chi_{A_{I_i}(x) }b_j\right)(x)\right|^\rho\right)^{1/\rho},$$
where $I_i:=(t_{i+1}, t_i]$ and
$$A_{I_i}(x):=\{y\in\rd:\,\, |x-y|\in I_i\}.$$
Then it suffices to show that
$$\left|\left\{x\in \rd\setminus \widetilde Q:\,\,
\left(\sum_i\left|\sum_j T\left(\chi_{A_{I_i}(x) }b_j\right)(x)\right|^\rho\right)^{1/\rho}>\frac\lz 4\right\}\right|
\ls \frac1\lz\|f\|_\lo.$$
For each $x\in \rd\setminus \widetilde Q$ and $I_i$, consider the following
two sets of indices $j's$:
$$L_{I_i}^1(x):=\{j:\,\, Q_j\subset A_{I_i}(x)\},\,\,\,
L_{I_i}^2(x):=\{j:\,\, Q_j\nsubseteq A_{I_i}(x),\, Q_j\cap A_{I_i}(x)\not=\emptyset\}.$$
Then we see that
\begin{eqnarray*}
\left(\sum_i\left|\sum_j T\left(\chi_{A_{I_i}(x) }b_j\right)(x)\right|^\rho\right)^{1/\rho}
&&\le \left(\sum_i\left|\sum_{j\in L^1_{I_i}(x)} T\left(\chi_{A_{I_i}(x) }b_j\right)(x)\right|^\rho\right)^{1/\rho}\\
&&\quad+\left(\sum_i\left|\sum_{j\in L^2_{I_i}(x)} T\left(\chi_{A_{I_i}(x) }b_j\right)(x)\right|^\rho\right)^{1/\rho}.
\end{eqnarray*}
By $(c_3)$, \eqref{e-cz kernel smooth cond} and the disjointness of $\{A_{I_i}(x)\}$,
\begin{eqnarray*}
&&\left(\sum_i\left|\sum_{j\in L^1_{I_i}(x)} T\left(\chi_{A_{I_i}(x) }b_j\right)(x)\right|^\rho\right)^{1/\rho}\\
&&\quad\le \sum_i \sum_{j\in L^1_{I_i}(x)}\int_\rd |K(x, y)-K(x, y_j)|\chi_{A_{I_i}(x)}(y)|b_j(y)|\,dy\\
&&\quad\le \sum_j \int_\rd |K(x, y)-K(x, y_j)||b_j(y)|\,dy\\
&&\quad\ls \sum_j \int_\rd \omega\left(\frac{|y-y_j|}{|x-y_j|}\right)
\frac1{|x-y_j|^d}|b_j(y)|\,dy,
\end{eqnarray*}
where $y_j$ is the center of $Q_j$. From this, \eqref{e-dini cond}, $(c_4)$ and $(c_5)$, we further deduce that
\begin{eqnarray*}
&&\left|\left\{x\in \rd\setminus \widetilde Q:\,\,
\left(\sum_i\left|\sum_{j\in L^1_{I_i}(x)} T\left(\chi_{A_{I_i}(x) }b_j\right)(x)\right|^\rho\right)^{1/\rho}>\frac\lz 8\right\}\right|\\
&&\quad\ls\frac1\lz \sum_j \int_{\rd}|b_j(y)|\,dy\int_{\rd\setminus \widetilde Q}
\omega\left(\frac{|y-y_j|}{|x-y_j|}\right)
\frac1{|x-y_j|^d}\,dx\\
&&\quad\ls \frac1\lz \sum_j \int_{\rd}|b_j(y)|\,dy\sum_{k=1}^\fz\omega(2^{-k})\ls \frac1\lz\|f\|_\lo.
\end{eqnarray*}

On the other hand, to estimate $(\sum_i|\sum_{j\in L^2_{I_i}(x)} T(\chi_{A_{I_i}(x)}b_j)(x)|^\rho)^{1/\rho}$, as
in \cite{CampbellJRW03TAMS}, we can assume that
for any $I_i\in \{I_i\}$, $I_i$ is either a dyadic interval $J_k:=(2^k, 2^{k+1}]$ for some $k\in\zz$ or a proper subset of a dyadic interval.
Accordingly, we split the set $\{i\}$ of indices into the following subsets:
\begin{align}\label{e-long short vari defn}\bigcup_{k\in\zz}\mathcal S_k:=\bigcup_{k\in\zz}\left\{i:\,\, I_i\subseteq J_k\right\}.
\end{align}
Note that if $j\in L_{I_i}^2(x)$, then $Q_j\cap \partial A_{I_i}(x)\not=\emptyset$, where $\partial A_{I_i}(x)$ is the boundary of $A_{I_i}(x)$.

Moreover, for each $N\in \zz$, let $\mathcal D_N(\rd)$ be the subset of $\mathcal D(\rd)$ of dyadic cubes having side-length $2^N$. Since $\rho>2$ as in the definition of the variation in \eqref{def variation}, now we have
\begin{eqnarray}\label{eee1}
&&\left|\left\{x\in \rd\setminus \widetilde Q:\,\,
\left(\sum_{i\in \mathbb N}\left|\sum_{j\in L^2_{I_i}(x)} T\left(\chi_{A_{I_i}(x) }b_j\right)(x)\right|^\rho\right)^{1/\rho}>\frac\lz {16}\right\}\right|\\
&&\quad\ls \frac1{\lz^2}\int_{\rd\setminus \wz Q}\sum_{i\in \mathbb N}
\left|\sum_{j\in \zz}\chi_{L^2_{I_i}(x)}(j) T\left(\chi_{A_{I_i}(x) }b_j\right)(x)\right|^2\,dx\nonumber\\
&&\quad=\frac1{\lz^2}\int_{\rd\setminus \wz Q}\sum_{k\in\zz}\sum_{i\in \mathcal S_k}\left|\sum_{N\in\zz} T\left(\sum_{Q_j\in \mathcal D_N(\rd)}\chi_{L^2_{I_i}(x)}(j)\chi_{A_{I_i}(x)}b_j\right)(x)\right|^2\,dx\nonumber\\
&&\quad\le \sum_{k\in\zz}\frac1{\lz^2}\int_{\rd\setminus \wz Q}\left|\sum_{N\in\zz}h_{k,\,N}(x)\right|^2\,dx,\nonumber
\end{eqnarray}
where for each $k$ and $N$, $\mathcal S_k$ is as in \eqref{e-long short vari defn} and
\begin{align}\label{e-hkN func defn}
h_{k,\,N}(x):=\left[\sum_{i\in\mathcal S_k}\left|
T\left(\sum_{Q_j\in \mathcal D_N(\rd)}\chi_{L^2_{I_i}(x)}(j)\chi_{A_{I_i}(x) }b_j\right)(x)\right|^2\right]^\frac12.
\end{align}

It remains to show that
\begin{align*}
\sum_{k\in\zz}\int_\rd\left|\sum_{N\in\zz}\chi_{\rd\setminus \wz Q}(x)h_{k,\,N}(x)\right|^2\,dx\ls \lz^2\sum_{N\in\zz}\|d_N\|_{L^2(\rd)}^2,
\end{align*}
where for each $N$,
$$d_N(x):=\sum_{Q_j\in \mathcal D_N(\rd)}\chi_{Q_j}(x).$$
By the well-known almost orthogonality lemma, it is sufficient to show that, for every $k$ and $N$,
\begin{align*}
\int_\rd\left|\chi_{\rd\setminus \wz Q}(x)h_{k,\,N}(x)\right|^2\,dx\ls \lz^22^{-|k-N|}\|d_N\|_{L^2(\rd)}^2.
\end{align*}

We now show that for any $x\in \rd\setminus \wz Q$,
\begin{align}\label{e-almost orthog h_{k,N}}
h_{k,\,N}(x)^2\ls \lz^22^{-|k-N|}\frac1{2^{kd}}
\int_{\{y:\,2^{k-1}\le |x-y|<2^{k+2}\}}d_N(y)\,dy.
\end{align}
Observe that for $i\in \mathcal S_k$ and  $Q_j\in \mathcal D_N(\rd)$, if $N> k-1-\log\sqrt d$
and $x\in\rd\setminus \wz Q$, then
we must have $Q_j\cap A_{I_i}(x)=\emptyset$; for otherwise,
for any $y\in Q_j\cap A_{I_i}(x)$,
$$4\sqrt d 2^N=\frac92\sqrt d2^N-\frac{\sqrt d}2 2^N\le |x-x_j|-|x_j-y|\le |x-y|\leq 2^{k+1},$$
which is impossible. This via \eqref{e-hkN func defn} implies that $h_{k,\,N}(x)=0$ for $N> k-1-\log\sqrt d$
and $x\in\rd\setminus \wz Q$. Thus, for $i\in\mathcal S_k$, we only need to prove
\eqref{e-almost orthog h_{k,N}} for $N\leq k-1-\log\sqrt d$.
For  each $i\in \mathcal S_k$ and $x\in\rd\setminus \wz Q$, let
$$g_N:=\sum_{j\in L^2_{I_i}(x):\,Q_j\in \mathcal D_N(\rd)}b_j \,\,{\rm and }\,\,P_i(x):=\bigcup_{j\in L^2_{I_i}(x):\, Q_j\in \mathcal D_N(\rd)}Q_j.$$
Then for any $y\in P_i(x)$, there exists $j\in L^2_{I_i}(x)$, $Q_j\in \mathcal D_N(\rd)$ containing $y$. Moreover, since $Q_j\cap A_{I_i}(x)\not=\emptyset$, there exists $z_j\in Q_j\cap A_{I_i}(x)$. By the fact $i\in \mathcal S_k$, we see
that
$2^k<|z_j-x|\le2^{k+1}$. This further implies that
\begin{align}\label{e-anlus-1}
|y-x|\ge |z_j-x|-|z_j-y|>2^k-\sqrt d 2^N\ge 2^{k-1}
\end{align}
and
\begin{align}\label{e-anlus-2}
|y-x|\le |z_j-x|+|z_j-y|< 2^{k+2}.
\end{align}
Therefore, we have that $P_i(x)\subset \{y: 2^{k-1}<|x-y|\le 2^{k+2}\}$ and hence,
$$h_{k,\,N}(x)^2\le \sum_{i\in \mathcal S_k}\left[\int_{P_i(x)}
|K(x,y)||g_N(y)|\,dy\right]^2\ls \frac1{2^{2kd}}\sum_{i\in \mathcal S_k}\left[\int_{P_i(x)}
|g_N(y)|\,dy\right]^2.$$
For each $i\in \mathcal S_k$, from $(c_4)$, we deduce that
\begin{align*}
\int_{P_i(x)}|g_N(y)|\,dy&\ls \sum_{j\in L^2_{I_i}(x):\,Q_j\in \mathcal D_N(\rd)} \int_{P_i(x)}
|b_j(y)|\,dy
\ls\sum_{j\in L^2_{I_i}(x):\,Q_j\in \mathcal D_N(\rd)} \lz |Q_j|\ls \lz 2^{(d-1)k+N},
\end{align*}
where the last inequality comes from the facts
that $i\in \mathcal S_k$ with $k>N+1+\log\sqrt d$, that
$Q_j\cap \partial A_{I_i}(x)\not=\emptyset$,
and that $\{Q_j\}$ are mutually disjoint.

Note that $\{I_i\}_{i\in\mathcal S_k}$ forms a partition of $(2^k, 2^{k+1}]$.
By using \eqref{e-anlus-1}, \eqref{e-anlus-2} and $(c_4)$, we now conclude that
\begin{align*}
h_{k,\,N}(x)^2&\ls \frac\lz{2^{kd}}2^{N-k}\sum_{i\in \mathcal S_k}\int_{P_i(x)}
|g_N(y)|\,dy\\
&\le\frac\lz{2^{kd}}2^{N-k}\sum_{i\in \mathcal S_k}\int_{P_i(x)}\sum_{j\in L^2_{I_i}(x):\,Q_j\in \mathcal D_N(\rd)}\left|b_j(y)\right|\,dy\\
&\ls\frac\lz{2^{kd}}2^{N-k}\sum_{j\in L^2_{I_i}(x)\,{\rm for\,some\,}i\in \mathcal S_k:\,Q_j\in \mathcal D_N(\rd)}\int_{\rd}\left|b_j(y)\right|\,dy\\
&\ls \frac{\lz^2}{2^{kd}}2^{N-k}\sum_{j\in L^2_{I_i}(x)\,{\rm for\,some\,}i\in \mathcal S_k:\,Q_j\in \mathcal D_N(\rd)}|Q_j|\\
&\ls \frac{\lz^2}{2^{kd}}2^{N-k}\int_{\{y:\,2^{k-1}<|x-y|\le2^{k+2}\}}d_N(y)\,dy.
\end{align*}
This shows \eqref{e-almost orthog h_{k,N}} and hence finishes the proof of Proposition  \ref{p-vari lp bdd imp weal type}.
\end{proof}

Under the assumption of Theorem \ref{p-sparse domi}, we see that $\mathcal{V}_{\rho}({\mathcal T}_{\ast})$ is of weak type (1,1) by Proposition \ref{p-vari lp bdd imp weal type}. Based on this fact, we then have the following pointwise estimates; see \cite{ChengMa18} for the proof.

\begin{lem} \label{l-vari pointwise domin}
For a.e. $x\in Q_0$,
$$ \mathcal{V}_{\rho}({\mathcal T}_{\ast}(f\chi_{3Q_0}))(x)\leq C\|\mathcal{V}_{\rho}({\mathcal T}_{\ast})\|_{L^1(\rd)\to L^{1,\infty}(\rd)} |f(x)| + \mathcal M_{\mathcal{V}_{\rho}({\mathcal T}_{\ast}),\,Q_0}f(x). $$
\end{lem}

Using some idea in \cite{GillespieTorrea04IsraelJM}, we have the following conclusion on the boundedness of the local grand maximal truncated operator
$\mathcal M_{\mathcal{V}_{\rho}({\mathcal T}_{\ast}),Q_0}$ as in Definition \ref{Gmaximal}.

\begin{prop}\label{p-grand maxi truncat bdd}
Let $Q_0$ be a fixed cube and $p\in(1,q]$ where $q$ is as in Theorem \ref{t-lp bdd vari matrix weigh}. Then $\mathcal M_{\mathcal{V}_{\rho}({\mathcal T}_{\ast}),Q_0}$
is bounded on $L^p(\mathbb R^d)$ and of weak type (1,1). Moreover,
$\|\mathcal M_{\mathcal{V}_{\rho}({\mathcal T}_{\ast}),Q_0}\|_{L^p(\rd)\to L^p(\rd)}$ and $\|\mathcal M_{\mathcal{V}_{\rho}({\mathcal T}_{\ast}),Q_0}\|_{L^1(\mathbb R^d)\to L^{1,\,\fz}(\rd)}$ are independent of
$Q_0$.
\end{prop}

\begin{proof}
We first show that $\mathcal M_{\mathcal{V}_{\rho}({\mathcal T}_{\ast}),Q_0}$
is bounded on $L^q(\rd)$. To this end, let $r\in(1, \min\{q, \rho\})$. By the
$L^q(\rd)$-boundedness of $\mathcal{V}_{\rho}({\mathcal T}_{\ast})$, it
suffices to show that for any $f\in L^q(\rd)$ and a. e. $x\in Q_0$,
\begin{align}\label{e-grand maxi trunc pointwise upp bdd}
\mathcal M_{\mathcal{V}_{\rho}({\mathcal T}_{\ast}),Q_0}f(x)
\le C\left[\cm\left(|f|^r\right)(x)\right]^{1/r}+\mathcal{V}_{\rho}({\mathcal T}_{\ast}f)(x)
\end{align}
with the positive constant $C$  independent of $Q_0$, $f$ and $x$.
Here $\mathcal Mf(x)$ is the  Hardy--Littlewood maximal function, defined as
\begin{align}\label{e-Hardy L maxi func defn}
\mathcal Mf(x):=\sup_{Q \ni x} {\frac1{|Q|}}\int_Q |f(y)|\,dy,
\end{align}
where the supremum is taken over all cubes $Q\subset \rd$.

To show \eqref{e-grand maxi trunc pointwise upp bdd}, for any $x\in Q_0$, $Q\ni x$, and $\xi\in Q$, let
\begin{align}\label{e-Bx defi}
B_x:=B(x, 9 dl(Q))\, {\rm and}\,\widetilde B_x:=B(x, 3\sqrt dl(Q_0)).
 \end{align}
 Then $3Q\subset B_x$ and
$3Q_0\subset \widetilde B_x$.
We write
\begin{eqnarray*}
\mathcal{V}_{\rho}\big({\mathcal T}_{\ast}(f\chi_{3Q_0\setminus3Q})\big)(\xi)&
&\le  \left|\mathcal{V}_{\rho}\big({\mathcal T}_{\ast}(f\chi_{3Q_0\setminus B_x})(\xi)
- \mathcal{V}_{\rho}\big({\mathcal T}_{\ast}(f\chi_{3Q_0\setminus B_x}))(x)\right|\\
&&\quad+\mathcal{V}_{\rho}\big({\mathcal T}_{\ast}(f\chi_{(B_x\cap3Q_0)\setminus 3Q})(\xi)+\mathcal{V}_{\rho}\big({\mathcal T}_{\ast}(f\chi_{3Q_0\setminus B_x}))(x)\\
&&=:{\rm I}+{\rm II}+{\rm III}.
\end{eqnarray*}
We first consider the term ${\rm II}$. From the definition of $\mathcal{V}_{\rho}\big({\mathcal T}_{\ast})$ and Minkowski's inequality, we see that
\begin{align*}
{\rm II}&\ls \sup_{\varepsilon_i\searrow0}
 \bigg(\sum_{i=1}^{\infty}\Big| \int_{\substack{\varepsilon_{i+1}<|\xi-y|\leq \varepsilon_i\\
 y\in (B_x\cap3Q_0)\setminus 3Q}} |K(\xi, y)||f(y)|\,dy \Big|^{\rho}\bigg)^{1/\rho}
\\&\ls \sup_{\varepsilon_i\searrow0}
  \int_{
 y\in (B_x\cap3Q_0)\setminus 3Q}  \Big(\sum_{i=1}^{\infty} \chi_{\{i:\ \varepsilon_{i+1}<|\xi-y|\leq \varepsilon_i\}}(i) |K(\xi, y)|^{\rho}|f(y)|^{\rho} \Big)^{1/\rho}\,dy
\\&\ls \sup_{\varepsilon_i\searrow0}
  \int_{
 y\in B_x\cap3Q_0, |\xi-y|\geq l(Q)}   |K(\xi, y)||f(y)| \,dy.
\end{align*}
Since $|B_x|\approx |Q|$, by the size condition of the kernel $K(x,y)$ as in \eqref{e-cz kernel size cond} and the
H\"older inequality, we obtain that
\begin{align*}
{\rm II}
&\ls \int_{B_x} {1\over |\xi-y|^d} |f(y)|\,dy\ls\frac1{|B_x|}\int_{B_x}|f(y)|\,dy\ls \left[\cm\left( |f|^r\right)(x)\right]^{1/r}.
\end{align*}

We now estimate the term ${\rm III}$.
By the sublinearity of $\mathcal{V}_{\rho}\big({\mathcal T}_{\ast}f)$,
we see that
\begin{align*}
{\rm III}&\le \mathcal{V}_{\rho}\big({\mathcal T}_{\ast}(f\chi_{\widetilde B_x\setminus B_x}))(x)+\mathcal{V}_{\rho}\big({\mathcal T}_{\ast}(f\chi_{\widetilde B_x\setminus 3Q_0}))(x)\\
&\le \mathcal{V}_{\rho}\big({\mathcal T}_{\ast}(f))(x)
+C\left[\cm\left( |f|^r\right)(x)\right]^{1/r},
\end{align*}
where the last inequality follows from using the definition of $\mathcal{V}_{\rho}\big({\mathcal T}_{\ast}f)$ for the first term, and from repeating the estimate in the term {\rm II} for the second term.

It remains consider the term {\rm I}. We claim that
\begin{align}\label{e-I pointwise upp bdd}
{\rm I}\ls \left[\cm\left( |f|^r\right)(x)\right]^{1/r}.
\end{align}
We write $\tilde f(y):=f(y)\chi_{3Q_0\setminus B_x}(y)$ and
\begin{align*}
{\rm I}&\le
\sup_{\varepsilon_j}\left\{\sum_j\left|\int_{\rd}\left[K(\xi, y)\chi_{\{\varepsilon_{j+1}<|\xi-y|\le \varepsilon_j\}}(y)-K(x, y)\chi_{\{\varepsilon_{j+1}<|x-y|\le \varepsilon_j\}}(y)\right]
\tilde f(y)\,dy\right|^\rho\right\}^{1/\rho}\\
&\le \sup_{\varepsilon_j}\left[\sum_j\left|\int_{\rd}[K(\xi, y)-K(x, y)]\chi_{\{\varepsilon_{j+1}<|\xi-y|\le \varepsilon_j\}}(y)
\tilde f(y)\,dy\right|^\rho\right]^{1/\rho}\\
&\quad+\sup_{\varepsilon_j}\left\{\sum_j\left|\int_{\rd}K(x, y)\left[ \chi_{\{\varepsilon_{j+1}<|\xi-y|\le \varepsilon_j\}}(y)-\chi_{\{\varepsilon_{j+1}<|x-y|\le \varepsilon_j\}}(y)\right]
\tilde f(y)\,dy\right|^\rho\right\}^{1/\rho}\\
&=:{\rm I}_1+{\rm I}_2.
\end{align*}
We first consider the term ${\rm I}_1$. Similar to the estimate of {\rm II} above, from the smoothness condition of the kernel $K(x,y)$ as in \eqref{e-cz kernel smooth cond}, we deduce that
\begin{align*}
{\rm I}_1\le\int_{\rd\setminus B_x}\left|K(\xi, y)-K(x, y)\right| |f(y)|\,dy
\ls \left[\cm\left( |f|^r\right)(x)\right]^{1/r}.
\end{align*}

To estimate ${\rm I}_2$, note that
$$\chi_{\{\varepsilon_{j+1}<|\xi-y|\le \varepsilon_j\}}(y)-\chi_{\{\varepsilon_{j+1}<|x-y|\le \varepsilon_j\}}(y)\not=0$$
if and only if at least one of the following four statements holds:
\begin{itemize}
 \item [{\rm  (i)}]$\varepsilon_{j+1}<|\xi-y|\le\varepsilon_{j}$ and $|x-y|\le \varepsilon_{j+1}$;
  \item [{\rm (ii)}]$\varepsilon_{j+1}<|\xi-y|\le\varepsilon_{j}$  and $|x-y|> \varepsilon_{j}$;
  \item[{\rm (iii)}] $\varepsilon_{j+1}<|x-y|\le\varepsilon_{j}$ and $|\xi-y|\le \varepsilon_{j+1}$;
  \item [{\rm (iv)}]$\varepsilon_{j+1}<|x-y|\le\varepsilon_{j}$ and $|\xi-y|> \varepsilon_{j}$.
\end{itemize}
This together with the fact that $|x-\xi|\le\sqrt dl(Q)$ implies the following four cases:
\begin{itemize}
 \item [{\rm  (i')}]$\varepsilon_{j+1}<|\xi-y|\le \varepsilon_{j+1}+\sqrt dl(Q)$;
  \item [{\rm (ii')}]$ \varepsilon_{j}<|x-y|\le \varepsilon_{j}+\sqrt dl(Q)$;
  \item[{\rm (iii')}] $\varepsilon_{j+1}<|x-y|\le \varepsilon_{j+1}+\sqrt dl(Q)$;
  \item [{\rm (iv')}]$\varepsilon_{j}<|\xi-y|\le \varepsilon_{j}+\sqrt dl(Q)$.
\end{itemize}
We further have that
\begin{align*}
&\left|\int_{\rd}K(x, y)\left[ \chi_{\{\varepsilon_{j+1}<|\xi-y|\le \varepsilon_j\}}(y)-\chi_{\{\varepsilon_{j+1}<|x-y|\le \varepsilon_j\}}(y)\right]
\tilde f(y)\,dy\right|\\
&\quad\le \int_{\rd}|K(x, y)|\chi_{\{\varepsilon_{j+1}<|\xi-y|\le \varepsilon_j\}}(y)\chi_{\{\varepsilon_{j+1}<|\xi-y|\le \varepsilon_{j+1}+\sqrt dl(Q)\}}(y)\left|\tilde f(y)\right|\,dy\\
&\quad\quad+\int_{\rd}|K(x, y)|\chi_{\{\varepsilon_{j+1}<|\xi-y|\le \varepsilon_j\}}(y)\chi_{\{\varepsilon_{j}<|x-y|\le \varepsilon_{j}+\sqrt dl(Q)\}}(y)\left|\tilde f(y)\right|\,dy\\
&\quad\quad+\int_{\rd}|K(x, y)|\chi_{\{\varepsilon_{j+1}<|x-y|\le \varepsilon_j\}}(y)\chi_{\{\varepsilon_{j+1}<|x-y|\le \varepsilon_{j+1}+\sqrt dl(Q)\}}(y)\left|\tilde f(y)\right|\,dy\\
&\quad\quad+\int_{\rd}|K(x, y)|\chi_{\{\varepsilon_{j+1}<|x-y|\le \varepsilon_j\}}(y)\chi_{\{\varepsilon_{j}<|\xi-y|\le \varepsilon_{j}+\sqrt dl(Q)\}}(y)\left|\tilde f(y)\right|\,dy\\
&\quad=:\sum_{k=1}^4{\rm I}_{2,\,k,\,j}.
\end{align*}
By similarity, we only estimate ${\rm I}_{2,\,1,\,j}$. If $\varepsilon_{j+1}< \sqrt dl(Q)$,
 we see that
$${\rm I}_{2,\,1,\,j}\le \int_{\rd\setminus B_x}|K(x,y)|\chi_{\{|\xi-y|< 2\sqrt dl(Q)\}}(y)|\tilde f(y)|\,dy=0.$$
If $\varepsilon_{j+1}\ge \sqrt dl(Q)$, by the H\"older inequality and \eqref{e-cz kernel size cond}, for $r\in(1, \min\{q,\rho\})$, we have
\begin{align*}
{\rm I}_{2,\,1,\,j}&\ls
\left[\int_{\rd}|K(x,y)|^r\left|\tilde f(y)\right|^r
\chi_{\{\varepsilon_{j+1}<|\xi-y|\le \varepsilon_j\}}(y)\,dy\right]^{1/r}
\left[\left(\varepsilon_{j+1}+\sqrt dl(Q)\right)^d-\varepsilon_{j+1}^d\right]^{1/{r'}}\\
&\ls\left[\int_{\rd}\frac{|\tilde f(y)|^r}{|x-y|^{rd}}\chi_{\{\varepsilon_{j+1}<|\xi-y|\le \varepsilon_j\}}(y)\,dy\right]^{1/r}
\left[\varepsilon_{j+1}^{d-1}l(Q)\right]^{1/{r'}}\\
&\ls \left[\int_{\rd}\frac{|\tilde f(y)|^r}{|x-y|^{r+d-1}}\chi_{\{\varepsilon_{j+1}<|\xi-y|\le \varepsilon_j\}}(y)\,dy\right]^{1/r}
\left[l(Q)\right]^{1/{r'}}.
\end{align*}
Since $r<\rho$, we then conclude that
\begin{align*}
\sup_{\{\varepsilon_j\}}\left(\sum_{j}{\rm I}^\rho_{2,\,1,j}\right)^{1/\rho}
&\ls \sup_{\{\varepsilon_j\}}\left(\sum_{j}
\left[\int_{\rd}\frac{|\tilde f(y)|^r}{|x-y|^{r+d-1}}\chi_{\{\varepsilon_{j+1}<|\xi-y|\le \varepsilon_j\}}(y)\,dy\right]^{\rho/r}\right)^{1/\rho}\left[l(Q)\right]^{1/{r'}}\\
&\le \sup_{\{\varepsilon_j\}}\left(\sum_{j}
\int_{\rd}\frac{|\tilde f(y)|^r}{|x-y|^{r+d-1}}\chi_{\{\varepsilon_{j+1}<|\xi-y|\le \varepsilon_j\}}(y)\,dy\right)^{1/r}\left[l(Q)\right]^{1/{r'}}\\
&\le \left(\int_{\rd\setminus B_x}\frac{|f(y)|^r}{|x-y|^{r+d-1}}\,dy\right)^{1/r}\left[l(Q)\right]^{1/{r'}}\\
&\ls \left[\cm\left(|f|^r\right)(x)\right]^{1/r}.
\end{align*}
Thus, we have
$${\rm I}_2\ls \left[\cm\left(|f|^r\right)(x)\right]^{1/r}.$$
This together with the estimate of ${\rm I}_1$ implies \eqref{e-I pointwise upp bdd}, and hence \eqref{e-grand maxi trunc pointwise upp bdd} holds.

By the Marcinkiewicz interpolation theorem, to finish the proof of Proposition
\ref{p-grand maxi truncat bdd}, it suffices to show that $\mathcal M_{\mathcal{V}_{\rho}({\mathcal T}_{\ast}),Q_0}$ is of weak type (1,1).
Moreover, as in the proof of Proposition \ref{p-vari lp bdd imp weal type}, for
any $f\in L^1(\rd)$ and $\lambda>0$, we apply the Calder\'on--Zygmund decomposition
to $f$ at height $\lambda$. Then we obtain functions $g$ and $b$ such that
$f=g+b=g+\sum_j b_j$ and the properties $(c_1)$-$(c_5)$ hold.
Moreover, let $\widetilde Q_0:=\cup_j 25\sqrt dQ_j$, where $\supp (b_j)\subset Q_j$.
The weak type (1, 1) of $\mathcal M_{\mathcal{V}_{\rho}({\mathcal T}_{\ast}),Q_0}$
is reduced to showing that
\begin{align*}
\left|\left\{x\in\rd\setminus \widetilde Q_0:\,\, \mathcal M_{\mathcal{V}_{\rho}({\mathcal T}_{\ast}),Q_0} b(x)>\frac\lambda 2\right\}\right|
\ls \frac{\|f\|_{L^1(\rd)}}{\lambda}.
\end{align*}
Observe that for any $x\in\rd$,
$$\mathcal M_{\mathcal{V}_{\rho}({\mathcal T}_{\ast}),Q_0} b(x)
\le \widetilde{\mathcal M}_{\mathcal{V}_{\rho}({\mathcal T}_{\ast}),Q_0} b(x)
+C\cm b(x),$$
where $B_x$ is as in \eqref{e-Bx defi} and
\begin{align*}
\widetilde{\mathcal M}_{\mathcal{V}_{\rho}({\mathcal T}_{\ast}),Q_0} b(x):= \left\{
                                                                  \begin{array}{ll}
                                                                    \sup\limits_{Q\ni x,\, Q\subset Q_0} \esssup\limits_{\xi\in Q} \mathcal{V}_{\rho}\big({\mathcal T}_{\ast}(b\chi_{3Q_0\setminus B_x})\big)(\xi), & {\ \ \ \ x\in Q_0;} \\[12pt]
                                                                    0, & \hbox{\ \ \ \ otherwise.}                                                                  \end{array}
                                                                \right.
\end{align*}
Then from the weak type (1,1) of $\cm$ and the definition of $\widetilde{\mathcal M}_{\mathcal{V}_{\rho}({\mathcal T}_{\ast}),Q_0}$, it suffices to show that
\begin{align}\label{e-grand maxi truc weak (1,1)}
\left|\left\{x\in Q_0\setminus \widetilde Q_0:\,\, \widetilde{\mathcal M}_{\mathcal{V}_{\rho}({\mathcal T}_{\ast}),Q_0} b(x)>\frac\lambda 4\right\}\right|
\ls \frac{\|f\|_{L^1(\rd)}}{\lambda}.
\end{align}
We now estimate $\widetilde{\mathcal M}_{\mathcal{V}_{\rho}({\mathcal T}_{\ast}),Q_0} b(x)$ for $x\in Q_0\setminus \widetilde Q$.
By the definition of $\widetilde{\mathcal M}_{\mathcal{V}_{\rho}({\mathcal T}_{\ast}),Q_0}$, we only need to consider the case $x\in Q_0\setminus \widetilde Q$.
For $x\in Q_0\setminus \widetilde Q$, take $Q\ni x$, $\xi\in Q$ and $\{\varepsilon_i\}_i$ such that
$$\widetilde{\mathcal M}_{\mathcal{V}_{\rho}({\mathcal T}_{\ast}),Q_0} b(x)
\le 2\left[\sum_i
\bigg|\sum_j T\left(\chi_{A_{I_i}(\xi)}\chi_{3Q_0\setminus B_x}b_j\right)(\xi)\bigg|^\rho\right]^{1/\rho},$$
where $I_i:=(\varepsilon_{i+1}, \varepsilon_i]$ and
$$A_{I_i}(\xi):=\{y\in\rd:\,\, |\xi-y|\in I_i\}.$$
For fixed $Q\ni x$, $\xi\in Q$, $\{\varepsilon_i\}$ and $I_i$, consider the following
three sets of indices $j's$:
$$L_{I_i}^1(\xi):=\{j:\,\, Q_j\subset A_{I_i}(\xi)\cap (3Q_0\setminus B_x)\},$$
$$L_{I_i}^2(\xi):=\{j:\, j\not\in L_{I_i}^1(\xi),\, Q_j\cap (A_{I_i}(\xi)\cap (3Q_0\setminus B_x))\not=\emptyset,\,Q_j\cap \partial(3Q_0)\not=\emptyset\},$$
\begin{align*}
L_{I_i}^3(\xi)&:=\{j:\, j\not\in L_{I_i}^1(\xi),\, Q_j\cap (A_{I_i}(\xi)\cap (3Q_0\setminus B_x))\not=\emptyset,\\
&\quad\quad \quad Q_j\cap \partial B_x\not=\emptyset\,\,{\rm or}\,\, Q_j\cap \partial(A_{I_i}(\xi))\not=\emptyset\}.
\end{align*}
Then we have
\begin{align*}
\left[\sum_i
\bigg|\sum_j T\left(\chi_{A_{I_i}(\xi)}\chi_{3Q_0\setminus B_x}b_j\right)(\xi)\bigg|^\rho\right]^{1/\rho}&\le
\sum_{m=1}^3\left[\sum_i
\bigg|\sum_{j\in L_{I_i}^m(\xi)} T\left(\chi_{A_{I_i}(\xi)}\chi_{3Q_0\setminus B_x}b_j\right)(\xi)\bigg|^\rho\right]^{1/\rho}\\
&=:\sum_{m=1}^3 L_m(\xi).
\end{align*}
From \eqref{e-cz kernel smooth cond}, the fact that $|\xi-y_j|\ge 2 |y-y_j|$
 for any $y\in Q_j$, and $\int_\rd b_j(y) dy=0$, it follows that
\begin{align*}
L_1(\xi)&\le \sum_i
\sum_{j\in L_{I_i}^1(\xi)}\int_\rd|K(\xi, y)-K(\xi, y_j)|\chi_{A_{I_i}(\xi)}(y)\chi_{3Q_0\setminus B_x}(y)|b_j(y)|\,dy\\
&\ls \sum_j\int_\rd\omega\left(\frac{|y-y_j|}{|\xi-y_j|}\right)
\frac1{|\xi-y_j|^d}|b_j(y)|\,dy\\
&\ls \sum_j\int_\rd\omega\left(\frac{c|y-y_j|}{|x-y_j|}\right)
\frac1{|x-y_j|^d}|b_j(y)|\,dy,
\end{align*}
where $y_j$ is the center of $Q_j$, and the implicit constant is independent of
$\{\varepsilon_i\}$,
$\xi$, $Q$ and $b_j$. Then by $(c_4)$, $(c_5)$ and \eqref{e-cz kernel smooth cond}, we see that
\begin{align*}
&\left|\left\{x\in Q_0\setminus \widetilde Q_0:\,\, L_1(\xi)>\frac\lambda {16}\right\}\right|\\
&\quad\ls\frac1\lambda \sum_j\int_\rd |b_j(y)|\,dy
\int_{\rd\setminus \widetilde Q_0} \omega\left(\frac{c|y-y_j|}{|x-y_j|}\right)
\frac1{|x-y_j|^d}\,dx\ls \frac1\lambda\|f\|_{L^1(\rd)}.
\end{align*}

To estimate $L_2(\xi)$ and $L_3(\xi)$, as in the proof of Proposition \ref{p-vari lp bdd imp weal type},
we may assume that $I_i$ is a subset of a dyadic interval $J_k:=(2^k, 2^{k+1}]$ for some $k\in\zz$.
Because for any $j\in L_2(\xi)\cup L_3(\xi)$, $Q_j\cap (A_{I_i}(\xi)\cap (3Q_0\setminus B_x))\not=\emptyset,$
then there exists $z\in Q_j\cap (A_{I_i}(\xi)\cap (3Q_0\setminus B_x))$.
Since $|x-\xi|\le \sqrt dl(Q)<\frac13|x-z|$,  we see that
\begin{align}\label{e-equiva distanc Q_j}
8\sqrt dl(Q_j)\le\frac23|x-z|<|x-z|-|x-\xi|\le|z-\xi|\le 2^{k+1}.
\end{align}
Moreover, for any $y\in Q_j\cap (3Q_0\setminus B_x)$,
\begin{align}
\label{e-equiva distanc}
|\xi-y|\approx |x-y|\approx|x-z|\approx |\xi -z|,
\end{align}

Now we estimate $L_2(\xi)$.
Note that there exists $u\in Q_j\cap \partial(3Q_0)$.
Then we have
$$l(Q_0)-2^{k-2}\le l(Q_0)-\sqrt dl(Q_j) \le|u-\xi|-|z-u|\le|z-\xi|\le 2^{k+1}$$
and
$$2^k<|z-\xi|\le|z-u|+|u-\xi|\le \sqrt d l(Q_j)+|u-\xi|<2^{k-2}+3\sqrt dl(Q_0).$$
Thus, $l(Q_0)\approx 2^k$. From this, $(c_4)$ and $(c_5)$,
 we conclude that
\begin{align*}
&\left|\left\{x\in Q_0\setminus \widetilde Q_0:\,\, L_2(\xi)>\frac\lambda {16}\right\}\right|\\
&\quad\ls\frac1\lambda\int_{Q_0\setminus \widetilde Q}
\sum_i\sum_{j\in L_{I_i}^2(\xi)}\int_\rd
\frac1{|Q_0|}\chi_{A_{I_i}(\xi)}(y)\chi_{3Q_0\setminus B_x}(y)|b_j(y)|\,dy\,dx\\
&\quad\ls\frac1\lambda\sum_j\int_{Q_0}
\int_\rd\frac1{|Q_0|}|b_j(y)|\,dy\,dx\ls \frac1\lambda\|f\|_{L^1(\rd)}.
\end{align*}

To estimate $L_3(\xi)$, for each $N\in \zz$, let $\mathcal D_N(\rd)$ be the subset of $\mathcal D(\rd)$ of dyadic cubes having side-length $2^N$.
For each $k\in\zz$, let
\begin{align}\label{e-vari defn}
\mathcal S_k:=\left\{i:\,\, I_i\subseteq J_k\right\}.
\end{align}

Since $\rho>2$ as in the definition of variation in \eqref{def variation}, now we have
\begin{eqnarray*}
&&\left|\left\{x\in Q_0\setminus \widetilde Q_0:\,\,
\left(\sum_{i\in\mathbb N}\left|\sum_{j\in L^3_{I_i}(\xi)} T\left(\chi_{A_{I_i}(\xi)}\chi_{3Q_0\setminus B_x}b_j\right)(\xi)\right|^\rho\right)^{1/\rho}>\frac\lz {16}\right\}\right|\\
&&\quad\ls \frac1{\lz^2}\int_{Q_0\setminus \wz Q_0}\sum_k\sum_{i\in\mathcal S_k}
\left|\sum_{j\in \zz}\chi_{L^3_{I_i}(\xi)}(j) T\left(\chi_{A_{I_i}(\xi)}\chi_{3Q_0\setminus B_x}b_j\right)(\xi)\right|^2\,dx\\
&&\quad=\frac1{\lz^2}\int_{Q_0\setminus \wz Q_0}\sum_{k\in\zz}\sum_{i\in\mathcal S_k}\left|\sum_{N\in\zz} T\left(\sum_{Q_j\in \mathcal D_N(\rd)}\chi_{L^3_{I_i}(\xi)}(j)\chi_{A_{I_i}(\xi)}\chi_{3Q_0\setminus B_x}b_j\right)(\xi)\right|^2\,dx\\
&&\quad\le \frac1{\lz^2}\sum_{k\in\zz}\int_{Q_0\setminus \wz Q_0}\left|\sum_{N\in\zz}h_{k,\,N}(\xi)\right|^2\,dx,
\end{eqnarray*}
where for each $k$ and $N$, $\mathcal S_k$ is as in \eqref{e-vari defn} and
\begin{align}\label{e-hkN func defn-2}
h_{k,\,N}(\xi):=\left[\sum_{i\in\mathcal S_k}\left|
T\left(\sum_{Q_j\in \mathcal D_N(\rd)}\chi_{L^3_{I_i}(\xi)}(j)\chi_{A_{I_i}(\xi)}\chi_{3Q_0\setminus B_x}b_j\right)(\xi)\right|^2\right]^\frac12.
\end{align}
It remains to show that
\begin{align*}
\sum_{k\in\zz}\int_\rd\left|\sum_{N\in\zz}\chi_{Q_0\setminus \wz Q}(x)h_{k,\,N}(\xi)\right|^2\,dx\ls \lz^2\sum_{N\in\zz}\|d_N\|_{L^2(\rd)}^2,
\end{align*}
where for each $N$,
$$d_N(x):=\sum_{Q_j\in \mathcal D_N(\rd)}\chi_{Q_j}(x).$$
By the well-known almost orthogonality lemma, it is sufficient to show that, for every $k$ and $N$,
\begin{align*}
\int_\rd\left|\chi_{Q_0\setminus \wz Q_0}(x)h_{k,\,N}(\xi)\right|^2\,dx\ls \lz^22^{-|k-N|}\|d_N\|_{L^2(\rd)}^2
\approx \lz^22^{-|k-N|}\|d_N\|_{L^1(\rd)}.
\end{align*}
To this end, it remains to show that
\begin{align}\label{e-almost orthog h_{k,N}-xi}
h_{k,\,N}(\xi)^2\ls \lz^22^{-|k-N|}\frac1{2^{kd}}
\int_{\{y:\,2^{k-2}\le |x-y|<2^{k+3}\}}d_N(y)\,dy.
\end{align}
Observe that for $i\in \mathcal S_k$ and  $Q_j\in \mathcal D_N(\rd)$, if $N\ge k-2-\log\sqrt d$, then
we must have $Q_j\cap A_{I_i}(\xi)=\emptyset$; for otherwise,
for any $y\in Q_j\cap A_{I_i}(\xi)$,
$$8\sqrt d 2^N< |x-y_j|-|\xi-x|-|y_j-y|\le |\xi-y_j|-|y_j-y|\le |\xi-y|\le2^{k+1},$$
which is impossible. This via \eqref{e-hkN func defn-2} implies that $h_{k,\,N}(\xi)=0$ for $N\ge k-2-\log\sqrt d$,
and \eqref{e-almost orthog h_{k,N}-xi} holds. Thus, for $i\in\mathcal S_k$, we only need to prove
\eqref{e-almost orthog h_{k,N}-xi} for $N<k-2-\log\sqrt d$.

We first claim that
\begin{align}\label{e-Q_j measure claim}
\sum_{j\in L^3_{I_i}(\xi):\,Q_j\in \mathcal D_N(\rd)} |Q_j|\ls 2^{(d-1)k+N},
\end{align}
Indeed, by $i\in \mathcal S_k$ with $k>N+2+\log\sqrt d$ and the fact that $\{Q_j\}$ are mutually disjoint, we have
\begin{align}\label{e-Q_j measure claim-1}
\sum_{\gfz{j\in L^3_{I_i}(\xi):\,Q_j\in \mathcal D_N(\rd)}{Q_j\cap \partial A_{I_i}(\xi)\not=\emptyset}}|Q_j|\ls 2^{(d-1)k+N}.
\end{align}
Moreover, if  $Q_j\cap (A_{I_i}(\xi)\cap (3Q_0\setminus B_x))\not=\emptyset$ and $Q_j\cap \partial B_x\not=\emptyset$, we see that
 there exists $z\in Q_j\cap (A_{I_i}(\xi)\cap (3Q_0\setminus B_x))$
satisfying \eqref{e-equiva distanc Q_j} and
$u\in Q_j\cap \partial B_x$ satisfying
$$l(Q)-2^{k-2}\le l(Q)-\sqrt dl(Q_j) \le|u-\xi|-|z-u|\le|z-\xi|\le2^{k+1}$$
and
$$2^k< |z-\xi|\le|z-u|+|u-\xi|\le \sqrt d l(Q_j)+|u-\xi|\le 2^{k-2}+10dl(Q).$$
Thus,
\begin{align*}
\frac3{40}d^{-1}2^{k}\le l(Q)\le 2^{k+2}.
\end{align*}
This via $\sqrt dl(Q_j)\le 2^{k-2}$  further implies that
\begin{align}\label{e-Q_j measure claim-2}
\sum_{\gfz{j\in L^3_{I_i}(\xi):\,Q_j\in \mathcal D_N(\rd)}{Q_j\cap  \partial B_x\not=\emptyset}}|Q_j|\ls 2^{(d-1)k+N}.
\end{align}
Therefore, \eqref{e-Q_j measure claim} follows from combining \eqref{e-Q_j measure claim-1} and \eqref{e-Q_j measure claim-2}
and hence the claim holds.

For  each $k$ and $i\in \mathcal S_k$, let
$$g_N:=\sum_{j\in L^3_{I_i}(\xi):\,Q_j\in \mathcal D_N(\rd)}b_j \,\,{\rm\ \ and\ \  }\,\,P_i(\xi):=\bigcup_{j\in L^3_{I_i}(\xi):\, Q_j\in \mathcal D_N(\rd)}Q_j.$$
Then for any $y\in P_i(\xi)$, there exists $j\in L^3_{I_i}(\xi)$, $Q_j\in \mathcal D_N(\rd)$ containing $y$. Moreover, assume that $z_j\in Q_j\cap A_{I_i}(\xi)\cap (3Q_0\setminus B_x))$. By the fact $i\in \mathcal S_k$, we see
that
$2^k<|z_j-\xi|\le2^{k+1}$. This further implies that
\begin{align}\label{e-anlus-xi-1}
|y-x|\ge|y-\xi|-|\xi-x|\ge |z_j-\xi|-|z_j-y|-|\xi-x|>2^{k-1}-\sqrt d 2^N\ge 2^{k-2}
\end{align}
and
\begin{align}\label{e-anlus-xi-2}
|y-x|\le|y-\xi|+|\xi-x|\le|y-z_j|+|z_j-\xi|+|\xi-x|< 2^{k+3}.
\end{align}
Therefore, we have that $P_i(\xi)\subset \{y: 2^{k-2}<|x-y|\le 2^{k+3}\}$ and hence,
$$h_{k,\,N}(\xi)^2\le \sum_{i\in \mathcal S_k}\left[\int_{P_i(\xi)}
|K(\xi,y)||g_N(y)|\,dy\right]^2\ls \frac1{2^{2kd}}\sum_{i\in \mathcal S_k}\left[\int_{P_i(\xi)}
|g_N(y)|\,dy\right]^2.$$
For each $i\in \mathcal S_k$, from $(c_4)$ and \eqref{e-Q_j measure claim}, we deduce that
\begin{align*}
\int_{P_i(\xi)}|g_N(y)|\,dy&\ls \sum_{j\in L^3_{I_i}(\xi):\,Q_j\in \mathcal D_N(\rd)} \int_{P_i(\xi)}
|b_j(y)|\,dy
\ls\sum_{j\in L^3_{I_i}(\xi):\,Q_j\in \mathcal D_N(\rd)} \lz |Q_j|\ls \lz 2^{(d-1)k+N}.
\end{align*}

Note that $\{I_i\}_{i\in\mathcal S_k}$ forms a partition of $(2^k, 2^{k+1}]$.
By using \eqref{e-anlus-xi-1}, \eqref{e-anlus-xi-2} and $(c_4)$, we now conclude that
\begin{align*}
h_{k,\,N}(\xi)^2&\ls \frac\lz{2^{kd}}2^{N-k}\sum_{i\in \mathcal S_k}\int_{P_i(\xi)}
|g_N(y)|\,dy\\
&\le\frac\lz{2^{kd}}2^{N-k}\sum_{i\in \mathcal S_k}\int_{P_i(\xi)}\sum_{j\in L^3_{I_i}(\xi):\,Q_j\in \mathcal D_N(\rd)}\left|b_j(y)\right|\,dy\\
&\ls\frac\lz{2^{kd}}2^{N-k}\sum_{j\in L^3_{I_i}(\xi)\,{\rm for\,some\,}i\in \mathcal S_k:\,Q_j\in \mathcal D_N(\rd)}\int_{\rd}\left|b_j(y)\right|\,dy\\
&\ls \frac{\lz^2}{2^{kd}}2^{N-k}\sum_{j\in L^3_{I_i}(\xi)\,{\rm for\,some\,}i\in \mathcal S_k:\,Q_j\in \mathcal D_N(\rd)}|Q_j|\\
&\ls \frac{\lz^2}{2^{kd}}2^{N-k}\int_{\{y:\,2^{k-2}<|x-y|\le2^{k+3}\}}d_N(y)\,dy.
\end{align*}
This shows \eqref{e-almost orthog h_{k,N}-xi} and hence finishes the proof of
Proposition \ref{p-grand maxi truncat bdd}.
\end{proof}

\begin{rem}\label{r-grand maxi traunc bdd}
Define the grand maximal truncated operator $\mathcal M_{\mathcal{V}_{\rho}({\mathcal T}_{\ast})}$ by
\begin{align*}
\mathcal M_{\mathcal{V}_{\rho}({\mathcal T}_{\ast})}f(x) := \sup_{Q\ni x} \esssup_{\xi\in Q} \mathcal{V}_{\rho}\big({\mathcal T}_{\ast}(f\chi_{\rd\setminus3Q})\big)(\xi) ,
\end{align*}
where the supremum is taken over all cubes $Q\subset \rd$ containing $x$.
From the proof of Proposition \ref{p-grand maxi truncat bdd},
we see that the conclusion of Proposition \ref{p-grand maxi truncat bdd}
also holds if we replace $\mathcal M_{\mathcal{V}_{\rho}({\mathcal T}_{\ast}),Q_0}$ with  $\mathcal M_{\mathcal{V}_{\rho}({\mathcal T}_{\ast})}$.
\end{rem}

Now we are ready to provide the proof for our first main result.
\begin{proof}[Proof of Theorem \ref{p-sparse domi}]
We first prove that there exist pairwise disjoint cubes  $\{P_j\}\subset \mathcal D(Q_0)$ such that $\sum_j |P_j|\le \varepsilon |Q_0|$ and
for a.\, e. $x\in Q_0$,
\begin{align}\label{e-sparse domin recur}
\mathcal{V}_{\rho}({\mathcal T}_{\ast}(f\chi_{3Q_0}))(x)\le C|f|_{3Q_0}
+\sum_{j}\mathcal{V}_{\rho}({\mathcal T}_{\ast}(f\chi_{3P_j}))(x)\chi_{P_j}(x).
\end{align}
By Proposition \ref{p-grand maxi truncat bdd}, we see that $\mathcal M_{\mathcal{V}_{\rho}({\mathcal T}_{\ast}),Q_0}$ is bounded from $\lo$ to $L^{1,\,\fz}(\rd)$. Now let
\begin{eqnarray*}
E:=&&\left\{x\in Q_0: \,\,|f(x)|>\alpha_d |f|_{3Q_0}\right\}\\
&&\bigcup
\left\{x\in Q_0:\,\,\mathcal M_{\mathcal{V}_{\rho}({\mathcal T}_{\ast}),Q_0}f(x)
> \alpha_d \|\mathcal M_{\mathcal{V}_{\rho}({\mathcal T}_{\ast}),Q_0}\|_{\lo\to L^{1,\,\fz}(\rd)}|f|_{3Q_0}\right\},
\end{eqnarray*}
where $\alpha_d:=2^{d+2}3^d\varepsilon^{-1}$. Then by the weak type (1, 1) of
$\mathcal M_{\mathcal{V}_{\rho}({\mathcal T}_{\ast}),Q_0}$ and the fact that $\supp(f)\subset Q_0$, we see that
\begin{align}\label{e-meas of E}
|E|\le \frac{\varepsilon}{2^{d+1}}|Q_0|.
\end{align}
Now we apply the Calder\'on--Zygmund decomposition to the function
$\chi_E$ at height $\lz:=\frac1{2^{d+1}}$. Then there exists a sequence
$\{P_j\}_j\subset \mathcal D(Q_0)$ of disjoint cubes such that
 \begin{align}\label{e-cz decom}
 \frac1{2^{d+1}}=\lz\le \frac1{|P_j|}\int_{P_j}\chi_{E}(x)\,dx\le 2^d \lz=\frac12
 \end{align}
and for a.\, e. $x\in Q_0\setminus \cup_j P_j$,  $\chi_E(x)\le \lz$.
These facts together with \eqref{e-meas of E} further imply that
$|E\setminus \cup_j P_j|=0$ and
\begin{align*}
\sum_{j}|P_j|\le 2^{d+1}\sum_{j}|P_j\cap E|\le 2^{d+1}|E|\le \varepsilon|Q_0|.
\end{align*}

Since $\supp(f)\subset Q_0$, we write
\begin{eqnarray*}
\mathcal{V}_{\rho}({\mathcal T}_{\ast}f)(x)\,\chi_{Q_0}(x)&=&
\mathcal{V}_{\rho}({\mathcal T}_{\ast}(f\chi_{3Q_0}))(x)\chi_{Q_0\setminus \cup_jP_j}(x)+ \sum_j\mathcal{V}_{\rho}({\mathcal T}_{\ast}(f\chi_{3Q_0}))(x)\chi_{P_j}(x)\\
&\le&\mathcal{V}_{\rho}({\mathcal T}_{\ast}(f\chi_{3Q_0}))(x)\chi_{Q_0\setminus \cup_jP_j}(x)+ \sum_j\mathcal{V}_{\rho}({\mathcal T}_{\ast}(f\chi_{3Q_0\setminus 3P_j}))(x)\chi_{P_j}(x)\\
&\quad& +\sum_j\mathcal{V}_{\rho}({\mathcal T}_{\ast}(f\chi_{3P_j}))(x)\chi_{P_j}(x).
\end{eqnarray*}
Since $|E\setminus \cup_j P_j|=0$, by Lemma \ref{l-vari pointwise domin} and the definition of $E$, we see that for  a.\, e. $x\in Q_0$,
\begin{align}\label{e-varia pointwise domin}
\mathcal{V}_{\rho}({\mathcal T}_{\ast}(f\chi_{3Q_0}))(x)\chi_{Q_0\setminus \cup_jP_j}(x)&&\ls \chi_{Q_0\setminus \cup_jP_j}(x)\left[|f(x)|+\mathcal M_{\mathcal{V}_{\rho}({\mathcal T}_{\ast}),Q_0}f(x)\right]\ls \varepsilon^{-1} |f|_{3Q_0}.
\end{align}
Moreover, from \eqref{e-cz decom} we deduce that for any $j$, $|P_j\cap E^c|\ge \frac12|P_j|$. This in turn implies that  there exists $x_j\in (P_j\cap E^c)$ such that
$$\mathcal M_{\mathcal{V}_{\rho}({\mathcal T}_{\ast}),Q_0}f(x_j)\ls\varepsilon^{-1} \|\mathcal M_{\mathcal{V}_{\rho}({\mathcal T}_{\ast})}\|_{\lo\to L^{1,\,\fz}(\rd)}|f|_{3Q_0}.$$
It then follows that for a.\, e. $x\in Q_0$,
\begin{align}\label{e-varia pointwise domin-2}
\mathcal{V}_{\rho}\left({\mathcal T}_{\ast}\left(f\chi_{3Q_0\setminus 3P_j}\right)\right)(x)\chi_{P_j}(x)\le \mathcal M_{\mathcal{V}_{\rho}({\mathcal T}_{\ast}),Q_0}f(x_j)\ls \varepsilon^{-1} |f|_{3Q_0}.
\end{align}
Combining this fact with \eqref{e-varia pointwise domin} and the fact that $\{P_j\}_j$ is mutually disjoint, we obtain \eqref{e-sparse domin recur}.

If $\overline {\mathcal F}:=\{R_j\}$ is a collection of disjoint dyadic subcubes of $Q_0$
that covers $\mathcal F=\{P_j\}$, then by the sublinearity of $\mathcal{V}_{\rho}({\mathcal T}_{\ast})$, we write
\begin{eqnarray*}
\mathcal{V}_{\rho}\left({\mathcal T}_{\ast}f\right)(x)
&\le&\mathcal{V}_{\rho}\left({\mathcal T}_{\ast}\left(f\chi_{3Q_0}\right)\right)(x)\chi_{Q_0\setminus \cup_jR_j}(x)+ \sum_j\mathcal{V}_{\rho}\left({\mathcal T}_{\ast}\left(f\chi_{3Q_0\setminus 3R_j}\right)\right)(x)\chi_{R_j}(x)\\
&\quad& +\sum_j\mathcal{V}_{\rho}\left({\mathcal T}_{\ast}\left(f\chi_{3R_j}\right)\right)(x)\chi_{R_j}(x).
\end{eqnarray*}

Since $\cup_j P_j\subset \cup_j R_j$, by the definition
of $E$, \eqref{e-varia pointwise domin} and \eqref{e-varia pointwise domin-2} still
hold with $P_j$ replaced with $R_j$. This shows (iii) and finishes the proof of Theorem \ref{p-sparse domi}.
\end{proof}

\section{Proof of Theorem \ref{t-lp bdd vari matrix weigh}}\label{s3}

In this section, by using Theorem \ref{p-sparse domi}, we first obtain a vector-valued version of domination of $\mathcal{V}_{\rho}({\mathcal T_n}_{\,,\,\ast}\vec f)(x)$ by convex body valued sparse operators in \cite{NazarovPTV17AdvMath}, then we further prove Theorem \ref{t-lp bdd vari matrix weigh}.
To begin with, we recall the so-called sparse collection of cubes in \cite{LernerNazarov15}, see also \cite{NazarovPTV17AdvMath}.

\begin{defn} \label{d-sparse}
Given $\eta\in(0, \infty)$, a collection $\mathcal S$ of cubes
(not necessarily dyadic) is said to be $\eta$-sparse provided that for every $Q\in\mathcal S$, there is a measurable subset $E_Q \subset Q$ such that
$|E_Q| \geq \eta |Q|$ and the sets $\{E_Q\}_{Q\in\mathcal S}$ are pairwise disjoint.
\end{defn}

Next we recall the convex body average of a vector $\vec f$ in \cite{NazarovPTV17AdvMath}. For $\vec f\in L^1(Q,\mathbb R^n)$, the convex body average $\llangle \vec f\rrangle_Q$ is defined as
\begin{align}\label{e-convex body aver defn}
\left\llangle \vec f\right\rrangle_Q:= \Big\{ \left\langle \varphi \vec f \right\rangle_Q\Big| \ \varphi: Q\to\mathbb R,\ \|\varphi\|_\lin\leq 1\Big\}.
\end{align}
Then $\llangle \vec f\rrangle_Q$ is a symmetric, convex, compact set in $\mathbb R^n$.  For a sparse family $\mathcal G$ of cubes,
in \cite{NazarovPTV17AdvMath} Nazarov et al introduced a sparse operator $\mathbb L:=\mathbb L_{\mathcal G}$ by
\begin{align*}
\mathbb L_{\mathcal G}(f):= \sum_{Q\in \mathcal G}  \left\llangle \vec f\right\rrangle_Q \chi_Q,
\end{align*}
where the sum is understood as Minkowsky sum. Moreover, from Lemma 2.5 in \cite{NazarovPTV17AdvMath}
we get that for a sparse family $\mathcal G$ and a compactly supported vector-valued function $\vec f\in L^1(\rd, \mathbb R^n)$, the set $\mathbb L_{\mathcal G}$ is a bounded convex symmetric subset of $\mathbb R^n$.

We also recall the John ellipsoids in \cite{NazarovPTV17AdvMath}.
An ellipsoid in $\mathbb R^n$ is an image of the closed unit ball $B$ in $\mathbb R^n$ under a non-singular affine transformation. Recall, that for a convex body (i.\,e. a compact convex set with non-empty interior) $K$ in $\mathbb R^n$, its John ellipsoid is an ellipsoid of maximal volume contained in $K$. It is known that the John ellipsoid is unique, and that if $K$ is symmetric, then its John ellipsoid $\mathcal E := \mathcal E_K$ is centered at $0$ and
$$  \mathcal E\subset K\subset \sqrt n\mathcal E;  $$
see \cite{NazarovPTV17AdvMath} or \cite{Howard97}.

We now recall the John ellipsoids for the set $\llangle \vec f\rrangle_Q$ as follows.

\begin{lem}[Lemma 2.6, \cite{NazarovPTV17AdvMath}]\label{l-convex body}
Let $\vec f \in L^1(Q,\mathbb R^n)$ be non-trivial (i.\,e. $\vec f(x)\not= 0$ on a set of positive measure). Then there exists a unique subspace $E \subset \mathbb R^n$ containing $\llangle \vec f\rrangle_Q$ such that $\llangle \vec f\rrangle_Q$ has non-empty interior in $E$.
\end{lem}

So, for a set $\llangle \vec f\rrangle_Q$ as in \eqref{e-convex body aver defn}, its John ellipsoid is defined as John ellipsoid in the subspace $E$ as in Lemma \ref{l-convex body}.

Under the assumption of Theorem \ref{t-lp bdd vari matrix weigh} that $\mathcal{V}_{\rho}({\mathcal T_n}_{\,,\,\ast})$ is bounded on $L^q(\mathbb R^d,\mathbb R^n)$ for some $q\in(1,\infty)$, we see that $\mathcal{V}_{\rho}({\mathcal T}_{\ast})$ is bounded on $L^q(\mathbb R^d, \mathbb R)$. Then, based on Theorem \ref{p-sparse domi} and the convex body sparse operator above, we have the
following result for $\mathcal{V}_{\rho}({\mathcal T_n}_{\,,\,\ast})$.

\begin{prop}\label{p-sparse domi 2}
Let $T$ be a Calder\'on--Zygmund operator as in Theorem \ref{t-lp bdd vari matrix weigh} and $T_n$ be as in \eqref{Tn}.
Then for any fixed cube $Q_0$,  $0<\delta<1$ and vector-valued functions $\vec f \in L^1(\mathbb R^d, \mathbb R^n)$
supported in $Q_0$, there exists a family $\mathcal G$ of disjoint dyadic subcubes of $Q_0$ such that
\begin{align} \label{sparse family}
\sum_{Q\in\mathcal G} |Q|\leq \delta |Q_0|
\end{align}
and
\begin{align}  \label{VT first generation}
\mathcal{V}_{\rho}\left({\mathcal T_n}_{\,,\,\ast}\vec f\ \right)(x) \in C\left\llangle \vec f\right\rrangle_{3Q_0} + \sum_{Q\in \mathcal G} \chi_Q(x) \mathcal{V}_{\rho}\left({\mathcal T_n}_{\,,\,\ast}\left(\vec f \chi_{3Q}\right)\right)(x)\quad {\rm\ a.\,e.\ on\ } Q_0,
\end{align}
where the constant $C$ depends only on $T$, $\delta$ and on the dimensions $n$ and $d$.
\end{prop}
\begin{proof}
Consider the representation of the John ellipsoid $\mathcal E$ of $\llangle \vec f\rrangle_{3Q_0}$
in principal axes, i.\,e. let $\vec e_1, \vec e_2,\ldots, \vec e_n$ be an orthonormal basis in $\mathbb R^n$ and
$\alpha_k \in [0, \infty)$ such that
\begin{align} \label{JE}
\mathcal E := \bigg\{ \sum_{k=1}^n x_k\alpha_k \vec e_k:\  x_k\in\mathbb R, \sum_{k=1}^n  x_k^2\leq1 \bigg\}.
\end{align}
Let $f_k(x) := \langle \vec f(x), \vec e_k \rangle_{\mathbb R^n}$. Since $\llangle \vec f\rrangle_{3Q_0} \subset \sqrt n \mathcal E$,
one can conclude that
$ |f_k|_{3Q_0}\leq \sqrt n\alpha_k $ by considering
$ \langle \varphi_kf_k \rangle_{3Q_0} $ with $\varphi_k := {\rm sgn} f_k$.

Applying Theorem \ref{p-sparse domi} with $\varepsilon := \delta n^{-1}$ to $T$ and to each $f_k$, we will get for each $f_k$ a collection $\mathcal G_k$ of dyadic subcubes of $Q_0$ such that a.\,e. $x$ on $Q_0$,
$$ \mathcal{V}_{\rho}({\mathcal T}_{\ast}f_k)(x)\leq C n\sqrt n \alpha_k + \sum_{Q\in\mathcal G_k} \chi_Q(x)\mathcal{V}_{\rho}({\mathcal T}_{\ast}(f_k\chi_{3Q}))(x), $$
where we used the estimate
$ |f_k|_{3Q_0}\leq \sqrt n \alpha_k$.
Let $\mathcal G$ be the collection of maximal cubes in the collection
$\cup_{k=1}^n\mathcal G_k$.
Since $\mathcal G$ covers any of $\mathcal G_k$, by using Theorem \ref{p-sparse domi} (iii), we have that
for a.\,e. $x\in Q_0$,
$$ \mathcal{V}_{\rho}({\mathcal T}_{\ast}f_k)(x)\leq Cn\sqrt n \alpha_k  + \sum_{Q\in\mathcal G} \chi_Q(x)\mathcal{V}_{\rho}({\mathcal T}_{\ast}(f_k\chi_{3Q}))(x). $$
Then clearly for a.\,e. $x\in Q_0$,
we have that
$$ \mathcal{V}_{\rho}({\mathcal T_n}_{\,,\, \ast}\vec f)(x) \in  C\sqrt n P + \sum_{Q\in\mathcal G} \chi_Q(x)\mathcal{V}_{\rho}\left({\mathcal T_n}_{\,,\,\ast}\left(\vec f\chi_{3Q}\right)\right)(x), $$
where $P$ is the box
$$ P:=\bigg\{ \sum_{k=1}^n x_k\alpha_k\vec e_k:\ x_k\in[-1,1] \bigg\}. $$

Since $P$ is contained in  $\sqrt n \mathcal E$,  where $\mathcal E$ is the John ellipsoid as defined in \eqref{JE},
we obtain that for a.\,e. $x\in Q_0$,
\begin{align*}
 \mathcal{V}_{\rho}({\mathcal T_n}_{\,,\, \ast}\vec f)(x) &\in  Cn^2 \mathcal E + \sum_{Q\in\mathcal G} \chi_Q(x)\mathcal{V}_{\rho}\left({\mathcal T_n}_{\,,\,\ast}\left(\vec f\chi_{3Q}\right)\right)(x) \\
 &
\subset   Cn^2 \left\llangle\vec f\right\rrangle_{3Q_0} + \sum_{Q\in\mathcal G} \chi_Q(x)\mathcal{V}_{\rho}\left({\mathcal T_n}_{\,,\,\ast}\left(\vec f\chi_{3Q}\right)\right)(x).
\end{align*}
In the end, by noting that
$$ \sum_{Q\in\mathcal G} |Q|\leq  \sum_{k=1}^n\sum_{Q\in\mathcal G_k} |Q| \leq n\cdot \delta n^{-1}|Q_0|=\delta |Q_0|, $$
we get that the proof of
Proposition \ref{p-sparse domi 2} is complete.
\end{proof}

\begin{prop}\label{p-sparse domi 3}
Let $T$ be a Calder\'on--Zygmund operator as in Theorem \ref{t-lp bdd vari matrix weigh} and $T_n$ be as in \eqref{Tn}.
Then there exists an $\eta$-sparse family $\mathcal F$ of cubes for some $\eta\in(0, 1)$, such that for every compactly supported
vector-valued function $\vec f\in L^1(\mathbb R^d,\mathbb R^n)$,
\begin{align}\label{T domi}
\mathcal{V}_{\rho}\left({\mathcal T_n}_{\,,\,\ast}\vec f\ \right)(x)\in C\mathbb L_{\mathcal F}(x),
\end{align}
where the constant $C$ depends only on the operator $T$ and dimensions $n$ and $d$.
\end{prop}
\begin{proof}

Take $\vec f\in L^1(\mathbb R^d,\mathbb R^n)$ and a cube $Q_0$ with $\supp (\vec f)\subset Q_0$. Applying Proposition \ref{p-sparse domi 2} with $\delta:={1\over2}$ we obtain a family
$\mathcal G_1$ of dyadic subcubes of $Q_0$ such that
\eqref{sparse family} and \eqref{VT first generation} hold.

Next, we apply  Proposition \ref{p-sparse domi 2} with $\delta:={1\over2}$ to each cube $Q$ in $\mathcal G_1$ (with function $\vec f\chi_{3Q}$) to get a family $\mathcal G_2$ and so on. Let $\mathcal G_0:=\{Q_0\}$ and $\mathcal G:= \cup_{\ell\geq0}\mathcal G_\ell$. Then we see that for any $N\in \nn$ and a.\, e. $x\in Q_0$,
\begin{align}\label{e-convex body repres N step}
\mathcal{V}_{\rho}\left({\mathcal T_n}_{\,,\,\ast}\vec f\ \right)(x) \in C\sum_{\ell=0}^N\sum_{Q\in \mathcal G_\ell}\chi_{Q}(x)\left\llangle \vec f\right\rrangle_{3Q}+
\sum_{Q\in \mathcal G_{N+1}}\chi_Q(x)\mathcal{V}_{\rho}\left({\mathcal T_n}_{\,,\,\ast}\left(\vec f \chi_{3Q}\right)\right)(x).
\end{align}
Observe that for any $\ell\ge 0$ and $Q\in \mathcal G_\ell$,
$$\sum_{P\in \mathcal G,\, P\subset Q}|P|=\sum_{k=1}^\fz \sum_{P\in \mathcal G_{\ell+k},\,P\subset Q}|P|+|Q|\le \sum_{k=1}^\fz\frac1{2^k} |Q|+|Q|=2|Q|.$$
Thus, by \cite[Lemma 6.3]{LernerNazarov15}, the family $\mathcal G$ is a dyadic ${1\over2}$-sparse family. Moreover, for any $\ell\ge 1$,
by the construction of $\mathcal G_\ell$,
$$\mathcal A_{\ell+1}:=\bigcup_{Q\in \mathcal G_{\ell+1}}Q\subset\mathcal A_{\ell}:=\bigcup_{Q\in \mathcal G_\ell}Q\subset Q_0,\,{\rm and \,}\sum_{P\in \mathcal G_{\ell+1},\,P\subset Q}|P|\le \frac12 |Q|\,\,{\rm for\, any\, }Q\in \mathcal G_{\ell}.$$
 Thus, since for each $\ell$, cubes in $\mathcal G_\ell$ are disjoint, we have
 $$\lim_{\ell\to\fz}|A_{\ell}|=\lim_{\ell\to\infty} \sum_{Q\in\mathcal G_\ell} |Q|\le
 \lim_{\ell\to\fz}\frac12\sum_{Q\in \mathcal G_{\ell-1}}|Q|\le\lim_{\ell\to\fz}\frac1{2^\ell} |Q_0|=0.$$
 This implies that for any $k\in\{1,2,\,\ldots, n\}$ and $\alpha>0$, the set
 $$E^k_\alpha:=\left\{x\in\rd: \lim_{N\to\fz}\sum_{Q\in\mathcal G_N}
 \chi_{Q}(x)\mathcal{V}_{\rho}(\mathcal T_{\ast}(f_k\chi_{3Q}))(x)>\alpha\right\}$$
 satisfies that $|E^k_\alpha|=0$. 

  By letting $N\to\fz$ in \eqref{e-convex body repres N step}, we conclude that for a.\,e. $x$ on $Q_0$,
\begin{align*}
   \mathcal{V}_{\rho}\left({\mathcal T_n}_{\,,\,\ast}\vec f\ \right)(x) \in C \sum_{Q\in\mathcal G} \left\llangle\vec f\right\rrangle_{3Q} \chi_Q(x).
\end{align*}

To dominate $\mathcal{V}_{\rho}({\mathcal T_n}_{\,,\,\ast}\vec f\ )$ outside of $Q_0$, for $\ell\geq0$, we apply Proposition \ref{p-sparse domi 2} to
$\vec f$ and $3^{\ell}Q_0$ instead of $Q_0$ therein, and see that for $x\in 3^{\ell+1}Q_0\backslash 3^\ell Q_0$,
$$ \mathcal{V}_{\rho}\left({\mathcal T_n}_{\,,\,\ast}\vec f\ \right)(x)\in C \left\llangle\vec f\right\rrangle_{3^{\ell+1}Q_0}.$$
So for a.\,e. $x\in \mathbb R^d$, we have
$$ \mathcal{V}_{\rho}\left({\mathcal T_n}_{\,,\,\ast}\vec f\ \right)(x) \in  C \sum_{Q\in\mathcal G} \left\llangle\vec f\right\rrangle_{3Q} \chi_Q(x) + C \sum_{\ell\geq1}  \left\llangle\vec f\right\rrangle_{3^{\ell}Q_0} \chi_{3^\ell Q_0}(x).  $$

Note that the above inclusion holds if we replace $\chi_Q$ by $\chi_{3Q}$. Next, since
the collection $\mathcal G$ is a dyadic ${1\over2}$-sparse family, the collection
$\{3Q:\ Q\in\mathcal G \}$ is an $\eta$-sparse family with $\eta:= {1\over 2\cdot 3^d }$.
If we add this collection to cubes $3^\ell Q_0$, $\ell\geq2$, it will remain $\eta$-sparse with the same $\eta$
as above.

So the collection
$$\mathcal F:=\{3Q:\ Q\in\mathcal G\}\bigcup\left\{ 3^\ell Q_0:\ \ell\geq2\right\}$$
is $\eta$-sparse, and hence
\eqref{T domi} holds.

The proof of
Proposition \ref{p-sparse domi 3} is complete.
\end{proof}

Based on Proposition \ref{p-sparse domi 3}, Theorem \ref{t-lp bdd vari matrix weigh} follows from combining the
proof of Theorem 1.14
in \cite{Cruz-UribeIM18IEOT} with $T\vec f$ replaced by $\mathcal{V}_{\rho}({\mathcal T_n}_{\,,\,\ast}\vec f\ )$, and then applying
Corollary 1.16 in \cite{Cruz-UribeIM18IEOT} in the end.
For the reader's convenience, we present the proof.

We begin with recalling some known results on the matrix weights in  \cite{Goldberg03PacificJM}. It is well known that for a matrix weight $W$, a cube $Q$ and any $1<p<\infty$, there exist positive definite matrices $W_Q$ and $W'_Q$, called reducing operators of $W^{\frac1p}$ and $W^{-\frac1p},$ respectively, such that
  \begin{align*}
  |Q|^{-{1\over p}}\left\|\chi_Q W^{1\over p} \vec{e}\right\|_{{L^p(\mathbb{R}^d, \mathbb{R}^n)}}\approx \left|W_Q\vec{e}\right|\quad{\rm and}\quad
          |Q|^{-{1\over p'}}\left\|\chi_Q W^{-{1\over p}} \vec{e}\right\|_{{L^{p'}(\mathbb{R}^d, \mathbb{R}^n)}}\approx \left|W'_Q\vec{e}\right|
          \end{align*}
for every $\vec{e}\in \mathbb{R}^n$. 
Note that $\|W_Q W'_Q\|\geq 1$ for any cube $Q$,
and that $W$ is a matrix $A_p$ weight if and only if
$ \|W_Q W'_Q\|\leq C<\infty $ for all cubes $Q$. We also mention
that if $W$ is in matrix $A_p$, then for every $\vec e\in \mathbb R^n$, $|W^\frac1p\vec e|^p$ is a scalar $A_p$ weight, and
\begin{align}\label{e-Ap weight matrix and scalar}
\left[\left|W^\frac1p\vec e\right|^p\right]_{A_p}\le[W]_{\mathcal A_p},
\end{align}
where $[W]_{\mathcal A_p}$ is as in \eqref{e-matrix ap defn}.

 For a scalar weight $w$, we say $w\in A_\fz$ if it satisfies the Fujii--Wilson condition
$$[w]_{A_\fz}:=\sup_Q\frac1{w(Q)}\int_Q\mathcal M(w\chi_Q)(x)\,dx<\fz,$$
where $\mathcal M$ is the Hardy--Littlewood maximal function as in \eqref{e-Hardy L maxi func defn}. Then for any $W\in \mathcal A_p$, using \eqref{e-Ap weight matrix and scalar} and the Fujii--Wilson condition, we
define
$$[W]_{\mathcal A^{sc}_{p,\,\fz}}:=\sup_{\vec e\in \mathbb R^n}\left[\left|W^\frac1p\vec e\right|^p\right]_{A_\fz},$$
$$\mathcal M_{r,\,p'}\vec f(x):=\sup_Q\left[\frac 1{|Q|}\int_Q\left|\vec f(y)\right|^{p'r}\,dy\right]^{\frac1{p'r}}\chi_Q(x),$$
and
$$\mathcal M_{s,\,p}\vec f(x):=\sup_Q\left[\frac1{|Q|}\int_Q|\vec f(y)|^{ps}\,dy\right]^{\frac1{ps}}\chi_Q(x),$$
where
\begin{align}\label{e-index r s defn}
r:=1+\frac1{2^{d+11}[W^{-\frac{p'}p}]_{\mathcal A^{sc}_{p',\,\fz}}},\,\,
s:=1+\frac1{2^{d+11}[W]_{\mathcal A^{sc}_{p,\,\fz}}}.
\end{align}
Then from \cite[Corollary 2.2]{Goldberg03PacificJM} (see also \cite{Cruz-UribeIM18IEOT}),
$$[W]_{\mathcal A^{sc}_{p,\,\fz}}\le [W]_{\mathcal A_p}.$$
By this and \cite[(5.3) and Lemma 2.4]{Cruz-UribeIM18IEOT}, we see that $W^{-\frac{p'}p}\in \mathcal A_{p'}$ and
\begin{align}\label{e-bdd r maxi fun}
\|\mathcal M_{r,\,p'}\|_{L^{p}(\rd,\,\mathbb R^n)\to L^{p}(\rd,\,\mathbb R^n)}&\ls (r')^\frac1p\approx \left[W^{-\frac{p'}p}\right]^{\frac1p}_{\mathcal A^{sc}_{p',\,\fz}}
\ls [W]_{\mathcal A_p}^{\frac1{p(p-1)}}.
\end{align}
Also, we have
\begin{align}\label{e-bdd s maxi fun}
\|\mathcal M_{s,\,p}\|_{L^{p'}(\rd,\,\mathbb R^n)\to L^{p'}(\rd,\,\mathbb R^n)}\ls (s')^\frac1{p'}\approx \left[W\right]^{\frac1{p'}}_{\mathcal A^{sc}_{p,\,\fz}}\ls \left[W\right]^{\frac1{p'}}_{\mathcal A_p}.
\end{align}

As in the proof of \cite[Theorem 1.10]{Cruz-UribeIM18IEOT}, to show Theorem \ref{t-lp bdd vari matrix weigh}, it suffices to show that for any $\vec f\in L^p(\rd, \mathbb R^n)$,
\begin{align}\label{e-lp bdd varia unweighted}
\left\|W^{\frac1p}\mathcal{V}_{\rho}\left({\mathcal T_n}_{\,,\,\ast}\left(W^{-\frac1p}\vec f\right)\right)\right\|_{L^p(\rd,\,\mathbb R^n)}\ls [W]_{A_p}^{1+{1\over p-1} -{1\over p}}
\left\|\vec f\right\|_{L^p(\rd,\,\mathbb R^n)}.
\end{align}

Moreover, to show \eqref{e-lp bdd varia unweighted}, by Proposition \ref{p-sparse domi 3},
it remains to show that for any sparse family $\mathcal S$ and the operator $T^{\mathcal S}\vec f$ of the form
$$T^{\mathcal S}\vec f(x):=\sum_{Q\in \mathcal S}\frac 1{|Q|}
\int_Q \varphi_Q(x,y)\vec f(y)\,dy\chi_Q(x),$$
where for each cube $Q$, $\varphi_Q$ is a real valued function such that $\|\varphi_Q(x, \cdot)\|_{L^\fz(\rd)}\le 1$,
the estimate
\begin{align}\label{e-lp bdd spar ope}
\left\|W^{\frac1p}\left(T^{\mathcal S}\left(W^{-\frac1p}\vec f\right)\right)\right\|_{L^p(\rd)}\ls [W]_{\mathcal A_p}^{1+{1\over p-1} -{1\over p}}
\left\|\vec f\right\|_{L^p(\rd)}
\end{align}
holds, where the implicit constant is independent of $\mathcal S$ and $\vec f$.

To show \eqref{e-lp bdd spar ope}, by a standard approximation argument, it
suffices to show that for any $\vec f, \vec g\in L^\fz_c(\rd,\mathbb R^n)$,
\begin{align}\label{e-lp bdd dual}
{\rm I}:=\left|\left\langle W^{\frac1p} T^{\mathcal S}W^{-\frac1p}\vec f, \vec g\right\rangle_{L^2(\rd)}\right|\ls [W]_{\mathcal A_p}^{1+{1\over p-1} -{1\over p}}
\left\|\vec f\right\|_{L^p(\rd,\,\mathbb R^n)}\left\|\vec g\right\|_{L^{p'}(\rd,\,\mathbb R^n)}.
\end{align}
From the definition of $T^\mathcal S$, we deduce that
\begin{align*}
{\rm I}\le \sum_{Q\in \mathcal S}\int_\rd \chi_Q(x)
\left|\left\langle W^{\frac1p}(x)\frac 1{|Q|}
\int_Q \varphi_Q(x,y)W^{-\frac1p}(y)\vec f(y)\,dy, \vec g(x)\right\rangle_{\mathbb R^n}\right|\,dx.
\end{align*}
For each cube $Q$, let  $W^{p'}_Q$ and $W^p_Q$ be the reducing operators such that
for any vector $\vec e\in \mathbb R^n$,
\begin{align*}
\left|W^{p}_Q \vec e\right|\approx \left[\frac1{|Q|}\int_Q\left|
W^{-\frac1p}(x)\vec e\right|^{(p'r)'}\,dx\right]^{\frac1{(p'r)'}},\,\,
\left|W^{p'}_Q \vec e\right|\approx \left[\frac1{|Q|}\int_Q\left|
W^{\frac1p}(x)\vec e\right|^{(ps)'}\,dx\right]^{\frac1{(ps)'}},
\end{align*}
where $r$ and $s$ are as in \eqref{e-index r s defn}.
Then (2.2) in \cite{Cruz-UribeIM18IEOT} shows that for any cube $Q$,
\begin{align}\label{e-reducing ope equiv norm-1}
\left[\frac 1{|Q|}\int_Q\left\|\left(W^p_Q\right)^{-1}
W^{-\frac1p}(y)\right\|^{(p'r)'}\,dy\right]^{\frac1{(p'r)'}}
&=\left[\frac 1{|Q|}\int_Q\left\|W^{-\frac1p}(y)\left(W^p_Q\right)^{-1}
\right\|^{(p'r)'}\,dy\right]^{\frac1{(p'r)'}}\\
&\ls\left\|W^p_Q\left(W^p_Q\right)^{-1}\right\|=1,\nonumber
\end{align}
and similarly,
\begin{align}\label{e-reducing ope equiv norm-2}
\left[\frac1{|Q|}\int_Q\left\|\left(W^{p'}_Q\right)^{-1}W^{\frac1p}(x)\right\|^{(ps)'}
\,dx\right]^{\frac1{(ps)'}}
&\ls\left\|W^{p'}_Q\left(W^{p'}_Q\right)^{-1}\right\|=1.
\end{align}
For each $x\in Q$, we have that by the H\"older inequality and \eqref{e-reducing ope equiv norm-1},
\begin{align*}
&\left|\left\langle W^{\frac1p}(x)\frac 1{|Q|}
\int_Q \varphi_Q(x,y)W^{-\frac1p}(y)\vec f(y)\,dy, \vec g(x)\right\rangle_{\mathbb R^n}\right|\\
&\quad=\left|\left\langle W^{p'}_Q W^p_Q\frac 1{|Q|}
\int_Q \varphi_Q(x,y)\left(W^p_Q\right)^{-1}W^{-\frac1p}(y)\vec f(y)dy, \left(W^{p'}_Q\right)^{-1}W^{\frac1p}(x)\vec g(x)\right\rangle_{\mathbb R^n}\right|\\
&\quad\le\sup_Q\left\|W^{p'}_Q W^p_Q\right\|
\frac 1{|Q|}
\int_Q |\varphi_Q(x,y)|\left|\left(W^p_Q\right)^{-1}W^{-\frac1p}(y)\vec f(y)\right|\,dy\left|\left(W^{p'}_Q\right)^{-1}W^{\frac1p}(x)\vec g(x)\right|\\
&\quad\ls [W]_{\mathcal A_p}^\frac1p\frac 1{|Q|}
\int_Q\left\|\left(W^p_Q\right)^{-1}W^{-\frac1p}(y)\right\|\left|\vec f(y)\right|\,dy \left\|\left(W^{p'}_Q\right)^{-1}W^{\frac1p}(x)\right\|
\left|\vec g(x)\right|\\
&\quad\le [W]_{\mathcal A_p}^\frac1p\frac 1{|Q|}
\left[\int_Q\left\|\left(W^p_Q\right)^{-1}W^{-\frac1p}(y)\right\|^{(p'r)'}\,dy\right]^{\frac1{(p'r)'}} \left[\int_Q\left|\vec f(y)\right|^{p'r}\,dy\right]^{\frac1{p'r}}\\
&\quad\quad\times\left\|\left(W^{p'}_Q\right)^{-1}W^{\frac1p}(x)\right\|
\left|\vec g(x)\right|\\
&\quad\ls [W]_{\mathcal A_p}^\frac1p\left[\frac 1{|Q|}\int_Q\left|\vec f(y)\right|^{p'r}\,dy\right]^{\frac1{p'r}}\left\|\left(W^{p'}_Q\right)^{-1}W^{\frac1p}
(x)\right\|\left|\vec g(x)\right|,
\end{align*}
where in the third inequality, we use the following inequality obtained in \cite[(5.5)]{Cruz-UribeIM18IEOT}:
$$\sup_Q\left\|W^{p'}_Q W^p_Q\right\|\ls [W]_{\mathcal A_p}^\frac1p.$$
This together with \eqref{e-reducing ope equiv norm-2} further implies that
\begin{align*}
{\rm I}&\ls [W]_{\mathcal A_p}^{\frac1p}\sum_{Q\in \mathcal S}\left[\frac 1{|Q|}\int_Q\left|\vec f(y)\right|^{p'r}\,dy\right]^{\frac1{p'r}}
\int_Q\left\|\left(W^{p'}_Q\right)^{-1}W^{\frac1p}(x)\right\|
\left|\vec g(x)\right|\,dx\\
&\le [W]_{\mathcal A_p}^{\frac1p}\sum_{Q\in \mathcal S}\left[\frac 1{|Q|}\int_Q\left|\vec f(y)\right|^{p'r}\,dy\right]^{\frac1{p'r}} \left[\int_Q\left\|\left(W^{p'}_Q\right)^{-1}W^{\frac1p}(x)\right\|^{(ps)'}
\,dx\right]^\frac1{(sp)'}\left[\int_Q\left|\vec g(x)\right|^{ps}\,dx\right]^{\frac1{ps}}\\
&\ls [W]_{\mathcal A_p}^{\frac1p}\sum_{Q\in \mathcal S}|Q|\left[\frac 1{|Q|}\int_Q\left|\vec f(y)\right|^{p'r}\,dy\right]^{\frac1{p'r}} \left[\frac1{|Q|}\int_Q\left|\vec g(x)\right|^{ps}\,dx\right]^{\frac1{ps}}.
\end{align*}

For any cube $Q\in \mathcal S$, let $E_Q$ be the subset of $Q$
as in Definition \ref{d-sparse}.
Now by the H\"older inequality, \eqref{e-bdd r maxi fun} and \eqref{e-bdd s maxi fun}, we have
\begin{align*}
{\rm I}&\ls [W]_{\mathcal A_p}^{\frac1p}\sum_{Q\in \mathcal S}
|E_Q|\inf_{x\in Q}\mathcal M_{r,\,p'}\vec f(x)\inf_{x\in Q}\mathcal M_{s,\,p}\vec g(x)\\
&\ls [W]_{\mathcal A_p}^{\frac1p}\sum_{Q\in \mathcal S}
\int_{E_Q}\mathcal M_{r,\,p'}\vec f(x)\mathcal M_{s,\,p}\vec g(x)\,dx\\
&\ls[W]_{\mathcal A_p}^{\frac1p}\int_{\rd}\mathcal M_{r,\,p'}\vec f(x)\mathcal M_{s,\,p}\vec g(x)\,dx\\
&\ls[W]_{\mathcal A_p}^{\frac1p}\left[\int_{\rd}\left[\mathcal M_{r,\,p'}\vec f(x)\right]^p\,dx\right]^{\frac1p}\left[\int_{\rd}\left[\mathcal M_{s,\,p}\vec g(x)\right]^{p'}\,dx\right]^{\frac1{p'}}\\
&\ls [W]_{\mathcal A_p}^{1+\frac1{p-1}-\frac1p}\left[\int_{\rd}\left|\vec f(x)\right|^p\,dx\right]^{\frac1p}\left[\int_{\rd}|\vec g(x)|^{p'}\,dx\right]^{\frac1{p'}}.
\end{align*}
Thus, \eqref{e-lp bdd dual} holds and then the proof of Theorem \ref{t-lp bdd vari matrix weigh} is complete.

\bigskip
\bigskip

{\bf Acknowledgments:}
X. T. Duong and J. Li are supported by ARC DP 160100153 and Macquarie University Research Seeding Grant. D. Yang is supported by the NNSF of China (Grant No. 11871254) and
the NSF of Fujian Province of China (Grant No. 2017J01011).

\bigskip


\medskip

\smallskip

Xuan Thinh Duong, Department of Mathematics, Macquarie University, NSW, 2109, Australia.

\smallskip

{\it E-mail}: \texttt{xuan.duong@mq.edu.au}

\vspace{0.3cm}

Ji Li, Department of Mathematics, Macquarie University, NSW, 2109, Australia.

\smallskip

{\it E-mail}: \texttt{ji.li@mq.edu.au}

\vspace{0.3cm}

Dongyong Yang (Corresponding author),
 School of Mathematical Sciences,
         Xiamen University,
         Xiamen, 361005, P. R. China

{\it E-mail}: \texttt{dyyang@xmu.edu.cn}

\end{document}